\documentclass[a4paper]{amsart}
\usepackage{appendix}

\usepackage{amsfonts}
\usepackage{amsmath}
\usepackage{amsthm}
\usepackage{booktabs}

\usepackage{hyperref}
\usepackage{xspace}
\usepackage[utf8]{inputenc}
\usepackage{xcolor}
\usepackage{mathtools}
\usepackage{esvect}
\usepackage{enumerate}

\usepackage{dsfont}
\usepackage{cleveref}
\usepackage{bm}

\usepackage{commath}
\mathtoolsset{showonlyrefs=false}

\usepackage[style=alphabetic, backend=biber, url=false, doi=false, autocite=footnote]{biblatex}
\hypersetup{breaklinks=true}
\allowdisplaybreaks

\definecolor{dblue}{rgb}{0.,0.,0.55}

\newcounter{counter}
\counterwithin{counter}{section}
\theoremstyle{plain}
\newtheorem{thm}[counter]{Theorem}
\newtheorem{lem}[counter]{Lemma}
\newtheorem{cor}[counter]{Corollary}
\newtheorem{prop}[counter]{Proposition}
\theoremstyle{definition}
\newtheorem{example}[counter]{Example}
\newtheorem{defn}[counter]{Definition}
\newtheorem{rem}[counter]{Remark}

\newcommand{\mull}{\textbf{m}}

\newcommand{\diag}{\text{diag}}

\newcommand{\cov}{\text{Cov}}
\newcommand{\var}{\text{Var}}
\newcommand{\NN}{\mathbb{N}}
\newcommand{\RR}{\mathbb{R}}
\newcommand{\Rplus}{\mathbb{R}_{\geq 0}}

\newcommand{\ind}[1]{\bm{1}_{#1}}

 \newcommand{\indep}{\perp\!\!\!\!\perp}

\usepackage{yhmath}

\usepackage{comment}
\usepackage{xcolor}
\definecolor{aliceblue}{rgb}{0.94, 0.97, 1.0}
\definecolor{apricot}{rgb}{0.98, 0.81, 0.69}
\definecolor{antiquewhite}{rgb}{0.98, 0.92, 0.84}
\usepackage{xargs}
\usepackage[colorinlistoftodos,prependcaption,textsize=scriptsize,linecolor=orange,backgroundcolor=antiquewhite,bordercolor=orange]{todonotes}

\newcommandx{\info}[2][1=]{\todo[linecolor=blue,backgroundcolor=aliceblue,bordercolor=blue,#1]{#2}}

\makeatletter
\newcommand{\amsvect}{%
  \mathpalette {\overarrow@\vectfill@}}
\def\vectfill@{\arrowfill@\relbar\relbar{\raisebox{-3.81pt}[\p@][\p@]{$\mathord\mathchar"017E$}}}
\newcommand{\amsvectb}{%
  \mathpalette {\overarrow@\vectfillb@}}
  \newcommand{\vecbar}{%
  \scalebox{0.8}{$\relbar$}}
\def\vectfillb@{\arrowfill@\vecbar\vecbar{\raisebox{-4.35pt}[\p@][\p@]{$\mathord\mathchar"017E$}}}
\makeatother

\usepackage{accents}
\newlength{\dhatheight}

\makeatletter
\newcommand\definealphabetloop[3]{%
  \ifx\relax#3\expandafter\@gobble\else\expandafter\@firstofone\fi
  {\expandafter\providecommand\expandafter*\csname#1#3\endcsname{#2{#3}}%
   \definealphabetloop{#1}{#2}}%
}%
\newcommand\definealphabet[2]{%
  \definealphabetloop{#1}{#2}abcdefghijklmnopqrstuvwxyzABCDEFGHIJKLMNOPQRSTUVWXYZ\relax
}%
\definealphabet{b}{\mathbb}%
\definealphabet{c}{\mathcal}%
\definealphabet{f}{\mathbf}%
\definealphabet{r}{\mathrm}%
\definealphabet{s}{\mathscr}%
\definealphabet{t}{\mathtt}%
\makeatother

\begin{document}

\title{Moments of polynomial processes under affine resetting}
\date{\today}

\author{Johannes Assefa}
\address{Department of Mathematical Stochastics\\Dresden Technical University}
\email{johannes.assefa@tu-dresden.de}
\author{Martin Keller-Ressel}
\address{Department of Mathematical Stochastics\\Dresden Technical University}
\email{martin.keller-ressel@tu-dresden.de}
\author{Zbigniew Palmowski}
\address{Department of Analysis and Stochastic Processes\\Wroc{\l}aw University of Science and Technology}
\email{zbigniew.palmowski@pwr.edu.pl}
\thanks{Zbigniew Palmowski was partially supported by the National Science Centre (Poland): grant 2023/51/B/ST1/01270.}

%\subjclass[2010]{}

\begin{abstract}
We consider polynomial processes, i.e., continuous-time Markov processes that preserve polynomials without increasing degree, and augment them with resetting. Resetting means that at the jump times of an independent renewal process, the polynomial process is instantaneously reset to a new state, which can be chosen deterministically, randomly or in state-dependent fashion. We show that under certain types of resets (`affine resetting') the moments of the resulting process remain tractable and can be calculated by matrix exponentials (for Poissonian reset times) or by a matrix renewal equation (for more general reset times). Our framework encompasses many models that have been considered in the existing literature, such as growth-collapse processes with deterministic growth or with L\'evy inflow, and Brownian motion with resets to the origin, to a fixed distribution
or by multiplying by a fixed proportion.
\end{abstract}

\maketitle

\tableofcontents

\section{Introduction}
The resetting of stochastic processes has become a prominent modeling tool across various fields, including
searching for a static target \cite{Bel, Ben, White}, the transmission control protocol \cite{MR2023017}, queueing theory, reliability, and stochastic control (see also \cite{Bonomo22,Mori23,Evans20} and references therein).
In a standard framework, a time-homogeneous Markov process $X=(X_t)_{t\ge0}$ evolves on a state-space $S\subseteq \mathbb R^d$ and undergoes \emph{resetting} at a (often independent) sequence of reset times $(\tau_i)_{i=1}^\infty$, resulting in a new process $\widetilde X$. The reset is instantaneous, may be state-dependent, and may involve additional randomization. Mathematically, the reset at $\tau_i$ can be expressed as
\begin{align*}
	\widetilde X_{\tau_i} = r(\widetilde X_{\tau_i-},Z_i),
\end{align*}
where $r(x,z)$ is a deterministic function and each mark $Z_i$ provides additional (independent) randomization, if desired.
By $Z$ we denote the generic mark.
Most contributions of the existing literature study very particular combinations of base process $X$ and reset function $r(x,z)$. Evans et al. \cite{Evans11}, for example, examine Brownian motion in combination with resetting to the origin or to a fixed distribution.
Boxma et al. \cite{Boxma06} study growth-collapse processes which correspond to deterministic linear growth with resets (`collapse') to a randomized proportion of the pre-reset value see also \cite{MR2023017}.
Detailed analysis can be also done for the $\alpha$-stable process $X$ with fixed proportional resets; see \cite{10.1214/25-AAP2248}.

While much of the existing literature focuses on specific combinations of base processes and reset functions
this paper establishes a general framework for the class of polynomial processes with resetting.
The primary objective of this study is to analyze polynomial processes under affine resetting {\bf (PUARs)} $\widetilde X$ and to provide a unified methodology
for computing higher-order moments and cross-moments.
We also provide a representation of PUARs in terms of their semimartingale characteristics.

Specifically, this paper defines a PUAR based on the following structural assumptions:
\begin{itemize}
    \item the base process $X$ is a polynomial process in the sense of \cite{Cuchiero12};
    \item the reset function $r(x,z)$ is `affine', i.e. of the matrix-vector product form $r(x,z) = z_0 + \bar z \cdot x$, where
    the random mark $Z$ with realisation  $z= (z_0, \bar z)\in \RR^d \times \RR^{d \times d}$ has law $\mu$ with finite moments of all orders;
    \item the resets $(\tau_i)_{i=1}^\infty$ are independent from $X$ and have i.i.d inter-arrival times with distribution $F$.
    %\item the resulting process $\widetilde X$ after resetting being a polynomial process under affine resetting {\bf (PUAR)}.
\end{itemize}
Unless the inter-arrival times are exponentially distributed, the MPUR $\widetilde X$ is in general not a Markov process.
Nevertheless, in Lemma \ref{lem:regenerative} we show that it satisfies the important \emph{regenerative property}.

To describe our main results, we start from recalling the definition of polynomial process. The formal definition is given in Definition \ref{def:MPUR}.
%Let $ S $ is assumed to be a closed subset of $ \mathbb{R}^d $ with nonempty interior, and we write $ \mathcal{S} $ for its Borel $ \sigma $-algebra.
We denote by $\cP_n = \cP_n(S)$ the space of polynomials on $\bR^d$ of degree at most $n \in \bN$ restricted to $S\subseteq\bR^d$ and set $\cP=\cup_n\cP_n$.
The dimension of the space of polynomials $\cP_n$ is given by
$	D(n) = \dim \cP_n = \binom{d+n}{n}$.
Fixing a graded and monomial basis $H_n(x)$ in $\cP_n$,  any polynomial $p \in \cP_n$ can be represented as a vector $\vec{p}$ in the basis $H_n$, that is,
% e.g., $(x,y,z)\mapsto 3z + xy$ is represented as $\vec{p} = (0,0,0,3,0,1,0,0,0,0)$ in \eqref{eq:basis_H2}. Mathematically, the polynomial $p \in \cP_n$ and its basis representation $\vec p$ are connected by
\begin{equation*}
p(x) = H_n(x)^\top \vec p, \qquad x \in S.
\end{equation*}
An $S$-valued time-homogeneous Markov process $X$ with an extended generator $\cA$ and a semigroup $ (P_t)_{t \geq 0} $ is
called \emph{$m$-polynomial} if for
all $k\in\{1,\dots,m\}$, $p\in\cP_k$, and $x\in S$, $t\geq 0$ we have $P_t\lvert p\rvert(x)<\infty$ and $\cA(\cP_k)\subseteq\cP_k$.
If $X$ is $m$-polynomial for every $m\geq0$, then it is called \emph{polynomial}.
The key fact about this class of processes is
the existence of a matrix $A=\cA|_{\cP_k}\in \bR^{D(k) \times D(k)}$, such that $P_t|_{\cP_k}=e^{tA}$ for all $t\geq0$; cf.~\cite{Cuchiero12}.

A foundational observation of this work is demonstrating that an affine reset mechanism
maps polynomials to polynomials of the same degree, allowing the operation to be represented by a graded upper triangular matrix $R$, termed the reset matrix.
More precisely, we prove in this article that for the mapping
\begin{align*}
	\cR: \quad p(x) \mapsto \int_{\cZ} p(r(x,z)) \mu(dz)
\end{align*}
that maps $\cP_n(S)$ to $\cP_n(S)$, for any $n \in \bN$, there exist a $D(n) \times D(n)$-matrix $R=R_n$ such that
\begin{align*}	
	(\cR p)(x) = H_n(x)^\top R_n \vec p, \quad x \in S.
\end{align*}
%The matrix $R$ is a graded upper triangular matrix and we call it the \emph{reset matrix} associated to $(r,\cZ,\mu)$.

Leveraging above property, in this paper we derive a precise moment formula for the process
$\widetilde X$, that is, we show that
for any $n \in \NN_0$ there exists a kernel $K = K_n \colon\bR_+ \to \bR^{D(n) \times D(n)}$, such that
\begin{equation*}
	\bE_x[p(\widetilde X_t)] = H(x)^\top K(t) \vec p,\qquad \forall\,t\geq0, p \in \cP_n.
\end{equation*}
Further, the kernel $K$ is the unique locally bounded solution of the matrix renewal equation
	\begin{align*}
		K(t)=e^{tA} (1-F(t)) + \int_0^t e^{uA} R K(t-u)\,F(du), \quad t\geq0.
	\end{align*}

In the case of Poissonian resetting, that is, when $F(t)=1-e^{-\lambda t}$ for $\lambda >0$,
the kernel $K$ has an explicit form:
\begin{equation*}
    K(t) = \exp\left(t\big(A + \lambda (R - I)\big)\right).
    \end{equation*}
In this case the moment formula simplifies into
\begin{align*}
		\bE_x[p(\widetilde X_t)] = H(x)^\top e^{t(A+\lambda(R-I))}\vec p,
	\end{align*}
for all $p\in\cP_n, x\in S$ and $t \geq 0$.
Taking $\lambda\downarrow 0$ that corresponds to lack of resetting, one recovers the classical result for the polynomial processes:
	\begin{align*}		
		\bE_x[p(X_t)]=H(x)^\top e^{tA}\vec{p},\quad t\geq0;
	\end{align*}
see \cite{Cuchiero12}. % \label{thm:moment_formula_classic} (duplicate; defined at the theorem below)
We also analyse a few explicit examples, in particular, growth--collapse process with Poisson jump times and uniform cutoff and bivariate Brownian motion with total reset.

Another main goal of this work is deriving a cross-moment formula for joint moments of a PUAR $\widetilde X$ evaluated at two different time points $t$ and $t+h$.
Namely, we show in Theorem \ref{cor:cross moment 2} that for $k,n \in \NN_0$ and all $p\in\cP_n$, $q \in \cP_k$,
\begin{align*}
	\bE_x\!\left[ q(\widetilde X_t)p(\widetilde X_{t+h}) \right] = H_{n+k}(x)^\top C_{q,h}(t)\vec p,\quad t,h\geq0,
\end{align*}
where $C_{q,h}\colon\bR_+\to\bR^{D(n+k)\times D(n)}$ is the unique locally bounded kernel satisfying the renewal equation \eqref{eq:correct-cross-moment-renewal}.
Again a more explicit cross-moment formula is given in the case of Poissonian resetting; see
\eqref{eq:poissonian_cross} for details.

Finally, this paper investigates the stochastic representations of PUARs. We first consider the case where $X$ is a polynomial diffusion devoid of killing and explosion satisfying the SDE
	\begin{equation*}
		dX_t = \mathfrak{b}(X_t)\,dt + \sigma(X_t)\,dW_t,
	\end{equation*}
where $W$ is a $d$-dimensional Brownian motion, and where, for every $i,j \in \set{1, \dotsc, d}$, it holds that $\mathfrak{b}_i \in \cP_1$ and $\mathfrak{a}_{ij}  = (\sigma \sigma^\top)_{ij} \in \cP_2$.
We show in Theorem \ref{prop:repesent_reset}, that the polynomial diffusion $\widetilde X$ with affine reset mechanism
%$(t,\cZ,\mu)$, and the inter-arrival distribution $F$ having a density $f$,
satisfies the SDE
\begin{align*}	
	d\widetilde X_t =\mathfrak{b}(\widetilde X_t)\,dt+\sigma(\widetilde X_t)\,dW_t+\int_{\cZ}\left(r(\widetilde X_{t^-},z)-\widetilde X_{t^-}\right)\,\mathsf{M}(dt,dz),
\end{align*}
where
\begin{equation*}
\mathsf{M}(\omega; dt, dz) = \sum_{i=1}^\infty \delta_{\tau_i(\omega), Z_i(\omega)}(dt,dz).
\end{equation*}
Further, we show that for any polynomial semi-martingale $X$ the process $\widetilde X$ is also a special semimartingale and we identify its  characteristics $(\widetilde B, \widetilde C, \widetilde\nu)$.

The results of this paper extend the previous seminal results concerning polynomial processes (see \cite{Cuchiero12}) by adding
the special mechanism of the affine resetting. It appears that this new feature allows to keep the main properties
of polynomial processes at the cost of slightly more complex expressions. Consequently, many known applications of polynomial processes can be
be also analysed in the context of the PUARs; see \cite{Guida21, Filipovic17, Filipovic19} for more details.\par 
The paper is organised as follows.
\Cref{sec:preliminaries} introduces Markov processes under resetting, polynomial processes, and related notation. In Section \ref{sec:pol_under_res} we prove the moment and cross-moment formulas. Finally, in Section \ref{sec:stochastic_res} we provide the semimartingale
stochastic representations of polynomial diffusions under affine resetting.

\section{Preliminaries}\label{sec:preliminaries}
In this section we collect some technical information that will be used throughout the paper.
\subsection{Markov processes under resetting}\label{sec:MPUR}
In the following, we introduce a general mathematical framework to treat Markov processes undergoing resetting. While many different variations of such processes have been studied in the literature (see e.g. \cite{Boxma06, Evans20, Evans11}) and there is also a substantial related literature on renewal processes and semi-Markov processes (see e.g. \cite{asmussen2003applied, Harlamov08, limnios2012semi}), we could not find an existing framework that fits our needs exactly, which motivates this subsection.\\
We work on a standard probability space $(\Omega, \cF, \bP)$ and assume that the following objects are given:
\begin{itemize}
\item A c\`adl\`ag, time-homogenous Markov process $X = (X_t)_{t \geq 0}$ with state space $S\subseteq \bR^d$; represented by a family of laws $(\bQ_x)_{x \in S}$ on the Skorokhod space $\bD(\Rplus, S)$;
\item An initial distribution $\nu$  on $S$;
\item A distribution function $F$ on $\Rplus$ without atom at zero. This is the \emph{distribution of inter-arrival times} of resets;
\item A measurable function $r: S \times \cZ \to S, (x,z) \mapsto r(x,z)$, where the \emph{mark space} $\cZ$ is a Polish space. The function $r$ is called the \emph{reset function}.
\item A distribution $\mu$ on the mark space $\cZ$. We call the triplet $(r,\cZ,\mu)$ the \emph{reset mechanism}.
\end{itemize}
Before defining the new process $\widetilde X$ for $X$ under resetting, we briefly explain the construction in words: First, a sequence $(\tau_i)_{i=0}^\infty$ of reset times, independent of $X$, whose inter-arrival times are i.i.d.\ with distribution $F$, is constructed. Between the reset times $\widetilde X$ evolves according to the law of the Markov process $X$. At each reset time $\tau_i$, the process $\widetilde X$ undergoes instantaneous, state-dependent and randomized resetting to its new value
\begin{equation}\label{eq:reset}
\widetilde X_{\tau_i} = r(\widetilde X_{\tau_i-}, Z_i),
\end{equation}
where the $Z_i$ are i.i.d.\ $\cZ$-valued random marks with distribution $\mu$, representing additional randomization at reset times. We formalize this pathwise construction as follows:

\begin{defn}[Markov process under resetting]\label{def:MPUR} Set $\tau_0 = 0$ and $Y_0 \sim \nu$. Recursively and independently, conditionally on $Y_{n-1} = y$, sample for each $n \in \NN$
\begin{itemize}
\item inter-arrival time $W_{n} \sim F$;
\item Markov segment $X^{(n)} \sim \bQ_y$;
\item mark $Z_{n} \sim \mu$;
\end{itemize}
and set
\[Y_{n} = r(X_{W_{n}-}^{(n)}, Z_{n}).\]
From this recursion define the \emph{reset times}
\[\tau_n = W_1  + \dotsm + W_n\]
and the \emph{Markov process under resetting (MPUR)} $\widetilde X$
as
\[\widetilde X_t = X_{t - \tau_{n-1}}^{(n)}, \qquad \tau_{n-1} \le t < \tau_n.\]
Finally, we write
\[N_t = \sum_{i=1}^\infty\ind{\tau_i \le t}\]
for the reset counter, and define
\[\cG_t := \sigma\left((\widetilde X_s)_{s \le t}, (N_s)_{s \le t}, (Z_i \ind{\tau_i \le t})_{i=1}^\infty\right) \vee \cN,\]
where $\cN$ is the $\sigma$-algebra generated by all $\bP$-null sets. Thus $\bG = (\cG_t)_{t \geq 0}$ is the completion of the smallest filtration for which $\widetilde{X}$ is adapted, the $\tau_i$ are stopping times, and the mark $Z_i$ is revealed at time $\tau_i$.
\end{defn}

\begin{rem}% \label{rem:MPUR}
We make the following observations:
\begin{enumerate}[(i)]
\item Since $\widetilde X$ is piecewise defined from c\`adl\`ag segments, also the process $\widetilde X$ is c\`adl\`ag. The recursive construction is measurable, thus $\widetilde X$ is measurable as a map from $(\Omega, \cF)$ to $\bD(\Rplus, S)$. We denote the induced law on $\bD(\Rplus, S)$ by $\bP^\nu$. In particular, we write $\bP^x$ when $\nu = \delta_x$ for $x \in S$.
\item In the literature on regenerative processes, the process $(\tau_n, Y_n)_{n \in \NN_0}$ is often called the \emph{embedded renewal process}. Note that $\widetilde X_{\tau_n} = Y_n$ for all $n \in \NN_0$.
\end{enumerate}
There are two relevant special cases of resets:
\begin{enumerate}[(i)]
\item In the case of \emph{unrandomized resets}, the mark space $\cZ$, the distribution $\mu$ and the marks $(Z_i)_{i=1}^\infty$ are superfluous, and the reset mechanism can be reduced to $r(x,z) = r(x)$, depending on $x$ only.
\item In the case of \emph{state-independent resets}, only the marks, but not the current state, determine the reset. In this case, we can choose $\cZ = S$ and reduce the reset mechanism to $r(x,z) = z$, so that the new value after the reset is drawn directly from $\mu$. Such processes have been described in \cite{Avrachenkov12}.
\end{enumerate}
\end{rem}

\begin{example}\label{ex:MPUR}
    \begin{enumerate}[(i)]
\item Many \emph{growth-collapse} processes (see e.g. \cite{Boxma06}) fit into this framework. To obtain the `growth-collapse process with Poisson jump times and uniform cutoffs' from \cite[Sec.~4]{Boxma06}, we can set $S = \Rplus$, $X_t = X_0 + ct$ (deterministic growth at rate $c > 0$), $F(t) = 1-e^{-\lambda t}$ (exponential inter-arrival times), and $r(x,z) = xz$ where $\mu$ is uniformly distributed on $\cZ = [0,1]$. The latter resetting mechanism represents the `collapse' to a uniformly distributed proportion of the pre-collapse value.\\
To obtain a growth-collapse process with `L\'evy-inflow', as in \cite{Eliazar04}, the process $X$ simply has to be substituted by a L\'evy subordinator.
\item \emph{Brownian motion with reset to zero or to a fixed distribution $\mu$}, studied for example in \cite{Evans11, Evans20}, fits into our framework. Here $S = \RR^d$ and $X$ is $d$-dimensional Brownian motion. Inter-arrival times are exponentially distributed and we have the state-independent reset $r(x,z) = 0$ (reset to zero) or the randomized reset $r(x,z)= z$ to the given distribution $\mu$ on $\cZ = \RR^d$.

    \end{enumerate}
\end{example}
Unless the inter-arrival times are exponentially distributed, the MPUR $\widetilde X$ is generally no longer a Markov process.
Nevertheless, it satisfies the important \emph{regenerative property} demonstrated in the following lemma.
Hence, there exists an embedded Markov renewal process $(\widetilde X_{\tau_n}, \tau_n)$ and the PUAR $\widetilde X$ 
meets the definition of a semi-Markov process.
\begin{lem}[Regenerative Property] \label{lem:regenerative}
Let $\Phi$ be a bounded measurable function from $\bD(\Rplus, S) \times \Rplus^\NN$ to $\RR^d$. Then
\begin{equation}\label{eq:regenerative}
\bE^x \left[\Phi\big((\widetilde X_{\tau_n + t})_{t \geq 0}, (\tau_{n+k} - \tau_n)_{k \in \NN}\big) \mid \cG_{\tau_n}\right] = \bE^{\widetilde X_{\tau_n}} \left[\Phi\big((\widetilde X_t)_{t \geq 0}, (\tau_{k})_{k \in \NN}\big)\right]
\end{equation}
$\bP_x$-a.s. for all $n \in \NN_0$.
 \end{lem}
\begin{proof}
For $n \in \NN$, we write
\[C_n = (W_n, X^{(n)}, Z_n)\]
for the random variables defined in the $n$-th `cycle' of the recursion in Definition~\ref{def:MPUR}. The random vectors $C_n$ take values in $\mathsf{C} := \Rplus \times \bD(\Rplus, S) \times \cZ$. Moreover, define the completed discrete filtration
\[\cH_n := \sigma(Y_n, (C_k)_{1 \le k \le n}) \vee \cN.\]
By Definition~\ref{def:MPUR}, we have conditional independence (`future cycles are independent of the past, conditionally on $Y_n$')
\begin{equation}\label{eq:cond_indep}
(C_{n+k})_{k \in \NN} \indep_{Y_n} \cH_n, \qquad \forall\,n \in \NN_0,
\end{equation}
and also the coincidence of the conditional laws (`the law of the cycles, conditionally on $Y_n$, is always the same')
\begin{equation}\label{eq:cond_law}
\cL\left((C_{n+k})_{k \in \NN} \mid Y_n = y\right) = \cL\left((C_{k})_{k \in \NN} \mid Y_0 = y\right), \quad \forall\,n\in \NN_0.
\end{equation}
Also by Definition~\ref{def:MPUR}, at any given reset time $\tau_n$, the evolution of $\widetilde X$ after $\tau_n$ and future reset times depend only on $Y_n$ and future cycles $(C_{n+k})_{k \in \NN}$. Formally, there exists a measurable function
\[\Gamma: S \times \mathsf{C}^{\NN} \to \bD(\Rplus, S) \times \Rplus^\NN,\]
such that
\[\left((\widetilde X_{\tau_n + t})_{t \geq 0}, (\tau_{n+k} - \tau_n)_{k \in \NN})\right) = \Gamma\left(Y_n, (C_{n+k})_{k \in \NN}\right).\]
Concatenating $\Phi$ with $\Gamma$, we have
\begin{multline*}
\bE^x \left[\Phi\big((\widetilde X_{\tau_n + t})_{t \geq 0}, (\tau_{n+k} - \tau_n)_{k \in \NN}\big) \mid \cH_n\right] = \\
\bE^x \left[\Phi \circ \Gamma\big(Y_n, (C_{n+k})_{k \in \NN}\big) \mid \cH_n\right] = \bE^x \left[\Phi \circ \Gamma\left(Y_n, (C_{n+k})_{k \in \NN}\right) \mid Y_n\right] = \\
\bE^{Y_n} \left[\Phi \circ \Gamma\left(Y_0, (C_k)_{k \in \NN}\right)\right] =
\bE^{\widetilde X_{\tau_n}} \left[\Phi((\widetilde X_t)_{t \geq 0}, (\tau_{k})_{k \in \NN})\right],
\end{multline*}
where we have used \eqref{eq:cond_indep} in the second and \eqref{eq:cond_law} in the third equality. Finally, observe that $\widetilde X_{\tau_n} = Y_n$ $\bP_x$-a.s. and $\cG_{\tau_n} \subseteq \cH_n$ for all $n \in \NN_0$. Taking $\cG_{\tau_n}$-conditional expectations of the last equality and substituting $\widetilde X_{\tau_n}$ for $Y_n$, \eqref{eq:regenerative} follows.
\end{proof}

The core idea of our paper is to combine this reset construction with the assumptions that

\begin{itemize}
\item $X$ is a polynomial process (in the sense of \cite{Cuchiero12}), and
\item the resetting mechanism $r(x,z)$ is affine.
\end{itemize}
We summarize key results on polynomial processes below; the affine resetting mechanism is discussed in Section~\ref{sec:affine_reset}.

\subsection{Notation related to polynomials}\label{sec:polynomial_notation}
Before our review of polynomial processes, we introduce some notation related to polynomials that will be used throughout this paper. As above, $ S $ is assumed to be a closed subset of $ \mathbb{R}^d $ with nonempty interior, and we write $ \mathcal{S} $ for its Borel $ \sigma $-algebra. We denote by $\cP_n = \cP_n(S)$ the space of polynomials on $\bR^d$ of degree at most $n \in \bN$ restricted to $S\subseteq\bR^d$ and set $\cP=\cup_n\cP_n$. For any $n ,d \in \NN_0$, the space $\cP_n$ is a finite-dimensional vector space. A simple combinatorial argument shows that the dimension of the space of polynomials $\cP_n$ is given by
\begin{align*}
	D(n) = \dim \cP_n = \binom{d+n}{n}.
\end{align*}
Throughout this paper we fix a graded, monomial basis
\begin{equation}% \label{eq:basis}
H(x) = \left(1, h_1(x), h_2(x), \dotsc, \right)
\end{equation}
of $\cP$ and write $H_n$ for its truncation to $\cP_n$. Note that we use \emph{zero-based} indexing for this basis and for all related objects, such as vectors or matrices expressed in it. A basis of $\cP$ is called \emph{graded} if it is ordered by total degree; it is called \emph{monomial} if it consists of monomials only.  For $n = 2$ and $S \subseteq \RR^3$, a graded, monomial basis is
\begin{equation}
	\label{eq:basis_H2}
	H_2(x,y,z) = \left(1, x,y,z, x^2, xy, xz, y^2, yz, z^2\right)^\top,\quad(x,y,z)\in S.
\end{equation}
Note that any polynomial $p \in \cP_n$ can be represented as a vector $\vec{p}$ in the basis $H_n$, e.g., $(x,y,z)\mapsto 3z + xy$ is represented as $\vec{p} = (0,0,0,3,0,1,0,0,0,0)$ in \eqref{eq:basis_H2}. Mathematically, the polynomial $p \in \cP_n$ and its basis representation $\vec p$ are connected by
\begin{equation}
p(x) = H_n(x)^\top \vec p, \qquad x \in S.
\end{equation}
Similarly, any linear map $\cL: \cP_n \to \cP_n$ can be represented by a real $D(n) \times D(n)$-matrix $L_n$, and we have
\begin{equation}
(\cL\,p)(x) = H_n(x)^\top L_n \vec p, \qquad x \in S.
\end{equation}
We also use the following convention: If the degree $n$ of the polynomial $p$ is clear, then subscripts may be dropped from $H_n$, $L_n$, etc. and can be inferred by `matching dimensions'. With this convention, the above equation simply reads as $(\cL\,p)(x) = H(x)^\top L \vec p$.
Occasionally, the quotient spaces $\cP_n^* = \cP_n / \cP_{n-1}$ will be of interest. Such a space consists of the equivalence classes of all degree-$n$ polynomials, which only differ by a polynomial of degree at most $n-1$. The dimension of $\cP_n^*$ is
\begin{align*}
	D_*(n) = D(n) - D(n-1) = \binom{n+d-1}{d-1}.
\end{align*}
Finally, given $q \in \cP_m$, we introduce the \emph{multiplication mapping}
\begin{align*}
	\mull_q \colon\cP_n\to \cP_{n+m}, \quad p \mapsto pq.
\end{align*}
As a linear mapping, $\mull_q$ has a matrix representation in the suitably truncated basis $H$, i.e., a $D(n+m) \times D(n)$-matrix, which we denote by $M_q$. This matrix satisfies
\begin{equation}\label{eq:matrix_Mq}
p(x)q(x) = H_{n+m}(x)^\top M_q \vec p, \qquad x \in S.
\end{equation}
For example, for $q(x,y) = x + 2y - 1$, the basis representation of $M_q: \cP_1 \to \cP_2$ in $H_2(x,y)$ is given by
\begin{align*}
	M_q = \begin{pmatrix}-1 & 1 & 2 & 0 & 0 & 0\\
	0 & -1 & 0 & 1 & 2 & 0 \\
	0 & 0 & -1 & 0 & 1 & 2 \end{pmatrix}^\top.
\end{align*}

\subsection{Polynomial processes}
In this section we review the definition of a polynomial process, which is given in \cite{Cuchiero12} within the framework of time-homogeneous Markov processes, taking values in a closed set $S \subseteq \RR^d$ with non-empty interior. Since polynomial processes are c\`adl\`ag, we may assume that the process $X$ is an element of $\bD(\Rplus, S)$. It is characterized by a family of probability laws $(\bQ_x)_{x \in S}$ giving rise to a semigroup $ (P_t)_{t \geq 0} $ of transition operators via\footnote{From the context, it will be clear whether $\bE_x[.]$ refers to expectation with respect to the law $\bQ_x$ of $X$ or to the law $\bP_x$ of $\widetilde X$.}
\begin{align*}
	P_t f(x) \coloneqq \bE_x[f(X_t)], \quad x \in S,
\end{align*}
acting on all Borel measurable functions $ f \colon S \to \mathbb{R} $ for which the integral is well defined.
%As usual, for any $x \in S$, we denote by $\bP_x(.) = \bP(.|X_0 = x)$ the conditional probability, given that $X$ starts in $x$.
% As polynomial processes are Markov processes $ X=(X_t)_{t \geq 0} $ with the property that $ (f(X_t))_{t \geq 0} $ is a special semimartingale for all linear functionals $ f $ on $ \bR^d $, we can safely assume that the probability space $ \Omega $ is the space of c\`adl\`ag functions $ \omega \colon \mathbb{R}_+ \to S $, as noted in \cite[Section 2]{Cuchiero12}. Thus, $ X $ is understood as the coordinate process $ X_t(\omega) = \omega(t) $, and $ (\mathcal{F}^0_t) $ denotes the filtration generated by $ X $, with $ \mathcal{F}^0 = \bigvee_{t \geq 0} \mathcal{F}^0_t $. In the sequel, we consider a right-continuous filtration $ (\mathcal{F}_t)_{t \geq 0} $ satisfying $ \mathcal{F}^0_t \subseteq \mathcal{F}_t $ and $ \mathcal{F} = \bigvee_{t \geq 0} \mathcal{F}_t $. Finally, we assume that for each $ x \in S $, there exists a probability measure $ \mathbb{P}_x $ on $ (\Omega, \mathcal{F}) $ such that $ X $ is Markovian with respect to $ (\mathcal{F}_t)_{t\geq0} $ and has semigroup $ (P_t)_{t\geq0} $, cf. \cite[Chapter 3]{Rogers00}.
We remark that in contrast to \cite{Cuchiero12} we do not allow for killing or explosion of $X$ and therefore do not attach a `cemetery state' to $S$.\par
%In the following we review the definition of polynomial processes, present their governing backward equation, and outline the derivation of the moment formula in the Markovian case as in \cite{Cuchiero12}. Let $X$ be an $S$-valued time-homogenous Markov process.
 An operator $\cA$ is called \emph{extended generator} for $X$ if its domain $\cD_\cA$ consists of those Borel-measurable functions $u\colon S\to\bR$ for which there exists a function $\cA u$ such that $M^u$ defined by
	\begin{align*}
		M_t^u= u(X_t)-u(x)-\int_0^t\cA u(X_s)\,ds,\quad t\geq0,
	\end{align*}
	is a $\bQ_x$-local martingale for every $x\in S$. Note that on bounded Borel functions, the extended generator coincides with the classical generator of the semigroup $(P_t)_{t \geq 0}$.
\begin{defn}[Cf.~\cite{Cuchiero12}]\label{def:polynomial process}
	An $S$-valued time-homogeneous Markov process $X$ with extended generator $\cA$ is called \emph{$m$-polynomial} if for all $k\in\{1,\dots,m\}$, $p\in\cP_k$, and $x\in S$,
	\begin{align*}
		%,\quad\forall k\in\{1,\dots,m\}.
		P_t\lvert p\rvert(x)=\bE_x[\lvert p(X_t)\rvert]<\infty, \quad \text{and}\quad \cA(\cP_k)\subseteq\cP_k.
	\end{align*}
	If $X$ is $m$-polynomial for every $m\geq0$, then it is called \emph{polynomial}.
\end{defn}
	In the sequel, $X$ denotes an $m$-polynomial process with extended generator $\cA$ and $k\in\{1,\dots,m\}$. Using Markovian techniques and the fact that $ \cA$ leaves $ \mathcal{P}_k $ invariant, one obtains the following defining property of polynomial processes.

\begin{thm}[Moment formula, cf.~\cite{Cuchiero12}]\label{thm:moment_formula_classic}
	Let $X$ be $m$-polynomial with extended generator $\cA$.
	%fix $k\in\{1,\dots,m\}$, and let $H=(h_1,\dots,h_N)$ be a basis of $\cP_k$.
	For each $k\in\{1,\dots,m\}$ there exists a matrix $A\in \bR^{D(k) \times D(k)}$, such that $P_t|_{\cP_k}=e^{tA}$ for all $t\geq0$, i.e. for all $x\in S$ and $p\in\cP_k$,
	\begin{align}
		\label{eq:moment formula with A}
		P_t\,p(x)=\bE_x[p(X_t)]=H_k(x)^\top e^{tA}\vec{p},\quad t\geq0.
	\end{align}
\end{thm}
Examples of polynomial processes include Brownian motion, L\'evy processes (under suitable conditions on finiteness of moments), Gaussian and non-Gaussian Ornstein-Uhlenbeck processes, the Cox-Ingersoll-Ross process, all Pearson diffusions, and many stochastic volatility models used in financial mathematics, such as the Heston model, the Bates model, and the Barndorff-Nielsen-Shephard model, cf.~\cite{Cuchiero12}.
\begin{rem}	
\begin{enumerate}[(i)]

	\item With respect to the graded basis $H_n$, the matrix $A_n=\cA|_{\cP_n}$ not only maps $\cP_n$ to $\cP_n$, but also every submatrix $A_k=\cA|_{\cP_k}, 0 \le k \le n$ maps $\cP_k$ to $\cP_k$. Therefore $A_n$ is block upper triangular, with diagonal blocks indexed by total degree $k=0,\dots,n$ and of size $D_\ast(k)$. In particular, $D_\ast(0)=1$ and the degree-zero block always satisfies $A_{00}=0$, since $\cA(x\mapsto 1)=0$.
    \item If the dimension $d$ of $S$ is one, then $D_*(k) = 1$ for all $k \in \NN$ and the matrix $A$ is upper triangular.

    \item An extension of the moment formula is presented in \cite[Theorem 2.5]{Filipovic17}, where the finiteness of absolute moments in \Cref{def:polynomial process} is not required. In this case, the conditional moment formula holds, i.e. for $p\in\cP_k$,
    % finiteness of unconditional moments is not assumed \cite[Example 2.6]{Filipovic19}.
\begin{equation}\label{eq:polynomial_moment_classic}
    	\mathbb{E}[p(X_T) \mid \cF_t] = H_k(X_t)^\top e^{(T-t)A} \vec{p}, \quad t \leq T.
    \end{equation}
\end{enumerate}
%\begin{rem}
%	Notes on polynomial jump diffusions.\info{why did we choose the Markovian framework?}
%\end{rem}
\end{rem}

\

\subsection{Stochastic representations of polynomial processes}
In this section we review two stochastic representations of polynomial processes, which can be found in \cite{Cuchiero12} and \cite{Filipovic17}. The first provides an SDE in the diffusion case, while the second represents the process in terms of semimartingale characteristics. We will extend these representations to the case of processes with resetting in \Cref{sec:stochastic_res}.
\begin{prop}\label{prop:stochastic}
\begin{enumerate}[(i)]
\item Let $X$ be a polynomial diffusion (i.e., a diffusion process which is also a polynomial process), then it satisfies the SDE
	\begin{equation*}
		dX_t = \mathfrak{b}(X_t)\,dt + \sigma(X_t)\,dW_t,
	\end{equation*}
where $W$ is a $d$-dimensional Brownian motion, and where, for every $i,j \in \set{1, \dotsc, d}$, it holds that $\mathfrak{b}_i \in \cP_1$ and $\mathfrak{a}_{ij}  = (\sigma \sigma^\top)_{ij} \in \cP_2$.
\item Let $X$ be a polynomial process on $S$ without killing or explosion. Then, on any of the filtered probability spaces $(\Omega, \cF, \bF, \bP_x), x \in S$, the process $X$ is a special It\^o semimartingale, and its characteristics $(B, C, \nu)$ associated with the ``truncation function'' $\chi(\xi) =\xi$ satisfy
\begin{align}
	\label{eq:B}
	&B_{t,i}=\int_0^{t} \mathfrak{b}_i(X_{s})\,ds,\\
	\label{eq:C}
	& C_{t,ij}+\int_0^t\int_{\bR^d}\xi_i\xi_j \,\nu(ds, d\xi)= \int_0^{t} \mathfrak{a}_{ij}(X_{s}) \,ds,
\end{align}
where $\mathfrak{b}_i \in \cP_1$ and $\mathfrak{a}_{ij} \in \cP_2$, for all $i,j \in \set{1, \dotsc, d}$. Moreover, $C$ and $\nu$ can be written as
\begin{align}
	\label{eq:C_eta}
	C_{t,ij}=\int_{0}^{t}c_{s,ij}\,ds, \quad \nu(dt, d\xi)=K_t(d\xi)\,dt,
\end{align}
where $(c_{ij})_{i,j \leq d}$ is a predictable process and $K_t(\omega,d\xi)\coloneqq K(X_{t-}(\omega),d\xi)$ is a predictable random measure on $(\bR^d, \mathcal{B}(\bR^d))$, which satisfies
\begin{align}
	\label{eq:K}
	\int_{\bR^d} \xi^{\mathbf{k}} K_t(d\xi) =\sum_{|\mathbf{l}|=0}^{|\mathbf{k}|}\alpha_{\mathbf{l}} X_t^{\mathbf{l}}, \quad \textrm{for almost all } t \geq 0,
\end{align}
with some coefficients $\alpha_{\mathbf{l}}$, for every multi-index $\mathbf{k} \in \bN_0^d$ with $|\mathbf{k}| \ge 3$. 
\end{enumerate}
\end{prop}
\begin{proof}
Assertion (i) follows from \cite[Lemma~2.2]{Filipovic17} and assertion (ii) from \cite[Proposition~2.12]{Cuchiero12}.
\end{proof}

\section{Polynomial processes under affine resetting (PUARs)}\label{sec:pol_under_res}
\subsection{Affine reset mechanisms and the reset matrix}\label{sec:affine_reset}
Recall from Section~\ref{sec:MPUR} that the resets of a Markov process can be described by a triplet $(r,\cZ,\mu)$, its \emph{reset mechanism}. We introduce the following specialization:

\begin{defn}[Affine reset mechanism] \label{def:affine_reset} A reset mechanism $(r,\cZ,\mu)$ is called \emph{affine} if
\begin{enumerate}[(i)]
\item $\cZ \subseteq \RR^d \times \RR^{d \times d}$;
\item $r: S \times \cZ \to S$ is an affine function, i.e. of the form
\begin{equation}\label{eq:affine}
	r(x,z) = z_0 + \bar z \cdot x, \qquad z = (z_0,\bar z) \in \cZ,
\end{equation}
where $z \cdot x$ denotes a matrix-vector product;
\item the random mark $z = (z_0,\bar z) \in \cZ$ has law $\mu$ with finite moments of all orders, i.e.
\begin{align*}
	\int_{\cZ}\lVert z\rVert^l\,\mu(dz)<\infty\quad\text{for all } l\in\bN,
\end{align*}
where $\norm{\cdot}$ is the Euclidean norm or any equivalent norm on $\cZ$.
\end{enumerate}
\end{defn}

\begin{rem}
\begin{enumerate}[(i)]
\item Depending on the state space $S \subseteq \RR^d$ the requirement that $r$ maps $S \times \cZ$ to $S$ together with the affine form \eqref{eq:affine} may impose certain restrictions on the choice of $\cZ$.
\item Simpler reset mechanisms like reset to zero $r(x,z) = 0$ (which does not need a mark space $\cZ$ at all) or $r(x,z') = xz'$ with $z' \in \cZ' = [0,1]$ are also considered affine, since they fit into Definition~\ref{def:affine_reset} after suitably extending $\cZ$.
\end{enumerate}
\end{rem}

%In general, the condition that $r$ maps into $S$ will impose some requirements on $\mathrm{supp}\,F$.
\begin{lem}[The reset matrix]\label{lem:linear stochastic reset}
Let $(r,\cZ, \mu)$ be an affine reset mechanism. Then the mapping
\begin{align*}
	\cR: \quad p(x) \mapsto \int_{\cZ} p(r(x,z)) \mu(dz)
\end{align*}
maps $\cP_n(S)$ to $\cP_n(S)$, for any $n \in \bN$.  Therefore, in any given graded basis $H_n$ of $\cP_n$ it has a matrix representation, given by a $D(n) \times D(n)$-matrix $R_n$, i.e.,
\begin{align}
	\label{eq:R matrix}
	(\cR p)(x) = H_n(x)^\top R_n \vec p, \quad x \in S.
\end{align}
The matrix $R_n$ is a graded upper triangular matrix and we call $R_n$ the \emph{reset matrix} associated to $(r,\cZ,\mu)$.
%Furthermore, the matrix $R_n$ is strictly upper-triangular when the multi-indices of $H_n$ are arranged in graded lexicographic order and $\mathrm{supp}F\subseteq\cZ^\vee$ (lower triangular); conversely, $R$ is strictly upper-triangular when the multi-indices of $H_n$ are arranged in graded reverse lexicographic order and $\mathrm{supp}F\subseteq\cZ^\wedge$ (upper triangular).
\end{lem}
\begin{proof}
	Let $p\in\cP_n(S)$. Since $x \mapsto r(x,z)$ is an affine function, $x \mapsto p(r(x,z)) = p(z_0 + \bar z \cdot x)$ is again in $\cP_n$ for every $z = (z_0,\bar z) \in \cZ$. Thus,
\begin{align*}
	(\cR p)(x)=\int p(r(x,z))\,\mu(dz)
\end{align*}
is an average over polynomials in $\cP_n$, and therefore again an element of $\cP_n(S)$.
Moreover, $\cR\colon \cP_n\to \cP_n$ is a linear operator on a finite-dimensional vector space.  Therefore, it has a matrix representation $R_n$ in the basis $H_n$, such that \eqref{eq:R matrix} holds. The graded upper triangular form of $R_n$ follows from the fact that $\cR$ also maps every subspace $\cP_k$ of $\cP_n$ to itself.
\end{proof}
\begin{example}% \label{ex:reset matrix}
Note that in all examples below we use a \emph{graded, monomial basis} $H_n = (1, h_1, \dotsc, h_{D(n)-1})$, and that the basis, as well as all matrices use \emph{zero-based indexing}; see Section~\ref{sec:polynomial_notation} for further details on notation.
	\begin{enumerate}[(i)]
\item For resetting to zero, i.e., $r\equiv 0$, the reset matrix is
\begin{align*}
	(R_n)_{ij} = \begin{cases}1&\quad \text{if $i=j=0$}\\
	0 &\quad \text{else.}\end{cases}
\end{align*}
\item For resetting to a fixed distribution $\mu$ (i.e. $r(x,z) = z$, $\cZ = S$), the reset matrix is
\begin{align*}
	(R_n)_{ij} = \begin{cases} \int_S h_j(z) \mu(dz) &\quad \text{if $i=0$}\\
	0 &\quad \text{else.}\end{cases}
\end{align*}
\item For deterministic partial resetting, i.e., $r(x,z) = \alpha x$ for some fixed $\alpha \in [0,1]$ the reset matrix is a diagonal matrix given by
\begin{align*}
	R_n = \diag\left(\alpha^k \, I_{D_*(k)}, \quad k=0,\dotsc,n\right),
\end{align*}
i.e. the diagonal consists of powers $\alpha^k$, repeated according to the dimension of $\cP^*_k$, for each $k=0\dotsc n$.
\item For stochastic partial resetting, i.e., $r(x,z) = zx$ and $\mu$ a probability measure on $\cZ = [0,1]$ the reset matrix is a diagonal matrix given by
\begin{align*}
	R_n = \diag\left(\int_0^1 z^k \mu(dz) \, I_{D_*(k)}, \quad k=0,\dotsc, n\right),
\end{align*}
i.e. the diagonal consists of the moments of $\mu$, repeated according to the dimension of the subspace $\cP^*_k$, for each $k=0\dotsc n$.
\item For a more concrete example, consider stochastic partial resetting  on $S = \RR^2$ with $\mu$ the uniform distribution on $[0,1]$. Then
\[R_2 = \diag\left(1, \;1/2, \;1/2, \;1/3, \;1/3, \;1/3 \right),\]
because $\frac{1}{2}, \frac{1}{3}$ are the first and second moment of the uniform distribution. The number of repetitions are the dimensions $D_*(0) = 1, D_*(1) = 2, D_*(2) = 3$, which can be read from the graded basis
\[H_2(x,y) = \left(1, \;x, \;y, \;x^2, \;xy, \;y^2\right).\]

\end{enumerate}
\end{example}

\begin{lem}\label{lem:conditional_reset}
Let $\widetilde X$ be a Markov process under affine resetting with mechanism $(r,\cZ,\mu)$. Then, for any $p \in \cP_n$ and $i \in \NN$, we have
\begin{equation}% \label{eq:conditional_reset}
\bE_x\left[p(\widetilde X_{\tau_i}) \mid \cG_{\tau_i-}\right] = H_n(\widetilde X_{\tau_i-})^\top R_n \vec p,
\end{equation}
where $R_n$ is the reset matrix.
\end{lem}
\begin{proof}
Using \eqref{eq:reset}, independence of marks, and Lemma~\ref{lem:linear stochastic reset}, we have
\begin{align*}
\bE_x\left[p(\widetilde X_{\tau_i}) \mid \cG_{\tau_i-}\right] &= \bE_x\left[p(r(\widetilde X_{\tau_i-}, Z_i)) \mid \cG_{\tau_i-}\right] \\
&= \int_\cZ p(r(\widetilde X_{\tau_i-}, z)) \mu(dz)  \\
&= (\cR\,p)(\widetilde X_{\tau_i-}) = H_n(\widetilde X_{\tau_i-})^\top R_n \vec p. \qedhere
\end{align*}
\end{proof}

\subsection{The moment formula for PUARs}
We show our main result, the moment formula for polynomial processes under affine resetting:

\begin{thm}[Moment formula for PUARs]\label{thm:moment formula 2}
Let $\widetilde X$ be a polynomial process under affine resetting (PUAR). Then, for any $n \in \NN_0$, there exists a kernel $K = K_n \colon\bR_+ \to \bR^{D(n) \times D(n)}$, such that
\begin{equation}
	\label{eq:moment_form2}
	\bE_x[p(\widetilde X_t)] = H(x)^\top K(t) \vec p,\qquad \forall\,t\geq0, p \in \cP_n
\end{equation}
%i.e. the moment formula holds true for any $p\in\cP_n$, $x\in S$, and $t\geq0$,
and $K$ is the unique locally bounded solution of the matrix renewal equation
	\begin{align}
		\label{eq:renewal equation}
		K(t)=e^{tA} \overline F(t) + \int_0^t e^{uA} R K(t-u)\,F(du), \quad t\geq0,
	\end{align}
    where $A$ is the generator matrix of $X$, $R=R_n$ is the reset matrix associated to the reset mechanism, $F$ the inter-arrival time distribution of resets, and $\bar F = 1 - F$.
    \end{thm}
    An application of Lemma~\ref{lem:regenerative} yields:
\begin{cor}[Conditional moment formula for PUARs]% \label{cor:moment formula_cond}
    At every reset time $\tau_i, i \in \NN$, the conditional moment formula
\begin{equation}
		% \label{eq:moment_form_cond}
	\bE_x[p(\widetilde X_{\tau_i + t}) \mid \cG_{\tau_i}] = H(\widetilde X_{\tau_i})^\top K(t) \vec p,\qquad \forall\,t\geq0, p \in \cP_n
\end{equation}
holds.
\end{cor}
\begin{proof}[Proof of \Cref{thm:moment formula 2}]
	Let $p\in\cP_n$, $x\in S$, and $t\geq0$ be fixed.
%	For matrix-valued functions $U,V$ we write
%	\begin{align*}
%		(U*V)(t)\coloneqq \int_0^t U(u)V(t-u)\,du.
%	\end{align*}
	Set $G(t)\coloneqq \overline F(t)e^{tA}$, define $L(du)\coloneqq e^{uA}R\,F(du)$, 
	%$L(t)\coloneqq e^{tA}R F(t)$
	and $K^{(0)}(t)\coloneqq G(t)$, and for $m\in\bN$, set 
	\begin{align*}
		K^{(m)}(t)\coloneqq (L*K^{(m-1)})(t) = \int_0^t e^{uA} R K^{(m-1)}(t-u) F(du).
	\end{align*}
	We first show that $\sum_{m=0}^\infty K^{(m)}(t)$ converges absolutely. Fix $T>0$ and set
	\begin{align*}
		M_T\coloneqq \sup_{0\leq s\leq T}\|e^{sA}\|<\infty,\quad c_T\coloneqq M_T\|R\|<\infty
	\end{align*}
    for some matrix norm $\norm{\cdot}$.
	Then, for any $t\leq T$ and $m\in\bN$, %writing $s_m=\sum_{i=1}^mu_i$ and $N_t=\sum_{i\in\bN}\mathds 1_{\{\tau_i\leq t\}}$,
we get
	\begin{align*}
		\|K^{(m)}(t)\| &\leq M_Tc_T^m F^{*m}(t),
	\end{align*}
    where $F^{*m}(t)$ is the $m$-th convolution power of $F$; cf. \cite[Ch.~V \S 4]{Feller71}. Note that $F^{*m}(t)$ is the distribution function of $\tau_m = W_1 + \dotsm + W_m$. Therefore for any $\theta>0$, by Markov's inequality,
	\begin{align*}
		F^{*m}(t) = \mathbb{P}_x(\tau_m \le t) = \bP_x(e^{-\theta \tau_m}\ge e^{-\theta t}) \leq e^{\theta t}\left(\bE[e^{-\theta W_1}]\right)^m,\quad m\in\bN_0,
	\end{align*}
	where we used that $(W_i)_{i\in\bN}$ are i.i.d. Since $W_1>0$ a.s., it holds that $\bE[e^{-\theta W_1}]\to0$ as $\theta\to\infty$. Choosing $\theta$ such that $c_T\bE[e^{-\theta W_1}]<1$ yields
	\begin{align*}
		\sum_{m=0}^\infty \|K^{(m)}(t)\| &\leq M_T\sum_{m=0}^\infty c_T^m F^{*m}(t)\\
		&\leq M_T\left(1+e^{\theta t}\sum_{m=1}^\infty \left(c_T\bE[e^{-\theta W_1}]\right)^m\right) <\infty.
	\end{align*}
	Thus $K(t)\coloneqq \sum_{m=0}^\infty K^{(m)}(t)$ defines a finite matrix kernel.\par 
	We next identify $K(t)$ as a solution of \eqref{eq:renewal equation}. We claim that
	\begin{align}
		\label{eq:m_reset_contribution}
		\bE_x[p(\widetilde X_t)\mathds 1_{\{N_t=m\}}] =H(x)^\top K^{(m)}(t)\vec p,\quad m\in\bN_0.
	\end{align}
	For $m=0$, this follows from
    \begin{align*}
    \bE_x[p(\widetilde X_t)\mathds 1_{\{N_t=0\}}] &= \bE_x[p(\widetilde X_t)\mathds 1_{\{\tau_1 > t\}}] \\
    &= \bE_x[p(X_t)] \bP_x(\tau_1 > t) = H(x)^\top e^{At} \bar F(t).
    \end{align*}
    Here, we have used that $\widetilde X_t=X_t$ on $\{\tau_1>t\}$, the independence of $\tau_1$ and $X$, and \eqref{eq:moment formula with A}. We proceed by induction. Let $m\in\bN$ and suppose that \eqref{eq:m_reset_contribution} holds for $m-1$. Since $N_t=m$ implies $\tau_1\leq t$, and $\tau_1$ is $\cG_{\tau_1}$-measurable, we obtain
	\begin{align*}
		\bE_x[p(\widetilde X_t)\mathds 1_{\{N_t=m\}}] &= \bE_x[ p(\widetilde X_t)\mathds 1_{\{N_t=m\}}\mathds 1_{\{\tau_1\le t\}}] \\
		&= \bE_x\left[ \mathds 1_{\{\tau_1\le t\}} \bE_x[ p(\widetilde X_t)\mathds 1_{\{N_t=m\}} \mid \cG_{\tau_1} ] \right] \\
		%&=\bE_x[\mathds 1_{\{\tau_1\le t\}}\bE_{r(X_{\tau_1^-},Z_1)}[p(\widetilde X_{t-\tau_1})\mathds 1_{\{N_{t-\tau_1}=m-1\}}]],
		&= \bE_x\left[ \mathds 1_{\{\tau_1\le t\}} \bE_{\widetilde X_{\tau_1}} [p(\widetilde X_{t-\tau_1})\mathds 1_{\{N_{t-\tau_1}=m-1\}}]]\right],
	\end{align*}
	where we have used Lemma~\ref{lem:regenerative} in the last step. Using the induction hypothesis and Lemma~\ref{lem:conditional_reset}, we continue with
	%Conditioning on the first reset time and using the renewal property after $\tau_1$ gives,
	\begin{align*}
		\bE_x[p(\widetilde X_t)\mathds 1_{\{N_t=m\}}] &=\bE_x\left[\mathds 1_{\{\tau_1\le t\}}  H(\widetilde X_{\tau_1})^\top  K^{(m-1)}(t-\tau_1)\right] \vec p \\
        &= \bE_x\left[\mathds 1_{\{\tau_1\le t\}}  H(X_{\tau_1 - })^\top  R K^{(m-1)}(t-\tau_1)\right] \vec p \\
        &= H(x)^\top \bE_x\left[\mathds 1_{\{\tau_1\le t\}}  e^{A\tau_1}  R K^{(m-1)}(t-\tau_1)\right] \vec p \\
		&=H(x)^\top\left(\int_0^t e^{uA}R K^{(m-1)}(t-u)\,F(du)\right)\vec p\\
		&=H(x)^\top K^{(m)}(t)\vec p.
	\end{align*}
	Summing \eqref{eq:m_reset_contribution} over $m\in\bN_0$ and using the absolute convergence yields \eqref{eq:moment_form2}. Finally,
	\begin{align*}
		K(t) &=G(t)+\sum_{m=1}^\infty (L*K^{(m-1)})(t)\\
		&=G(t)+(L*K)(t),
	\end{align*}
	which is precisely \eqref{eq:renewal equation}.\par
	It remains to prove uniqueness. Let $\widetilde K$ be another locally bounded solution of \eqref{eq:renewal equation} and set $D=K-\widetilde K$. Then
    \begin{equation}\label{eq:D_convolution}
    D(t) = \int_0^t e^{uA} R\,D(t-u) F(du).
    \end{equation}
    For fixed $T>0$ set $D_T=\sup_{0\leq s\leq T}\|D(s)\|<\infty$. Choosing $\theta$ such that $c_T\bE[e^{-\theta W_1}]<1$ and iterating \eqref{eq:D_convolution}, we obtain
	\begin{align*}
		\|D(t)\| \leq D_Tc_T^m F^{*m}(t) = D_Tc_T^m \bP(\tau_m \le t) \leq D_Te^{\theta T}\left(c_T\bE[e^{-\theta W_1}]\right)^m,
	\end{align*}
	for all $t\leq T, m \in \NN$. Letting $m\to\infty$ gives $D(t)=0$, and uniqueness follows.
\end{proof}

Under Poissonian resetting the solution $K(t)$ of \eqref{eq:renewal equation} can be written explicitly as a matrix exponential as in the classical moment formula \eqref{eq:moment formula with A}.

\begin{cor}[Poissonian resetting]\label{cor:poisson_renewal_times}
	In the setting of Theorem~\ref{thm:moment formula 2}, assume Poissonian resetting at rate $\lambda > 0$, i.e., $F(t)=1 - e^{-\lambda t}$. Then
    \begin{equation}\label{eq:exponential_kernel}
    K(t) = \exp\left(t\big(A + \lambda (R - I)\big)\right),
    \end{equation}
and the moment formula becomes
\begin{align*}
		\bE_x[p(\widetilde X_t)] = H(x)^\top e^{t(A+\lambda(R-I))}\vec p
	\end{align*}
for all $p\in\cP_n, x\in S$ and $t \geq 0$.
\end{cor}
\begin{proof}
We have to show that $K(t) = \exp(t(A + \lambda(R - I)))$ solves \eqref{eq:renewal equation} when $F(t) = 1 - e^{-\lambda t}$; by Theorem~\ref{thm:moment formula 2} any such solution must be the unique solution. Subtracting the right-hand side from the left-hand side we obtain
\begin{align*}
D(t) &= K(t) - e^{t(A-\lambda I)} - \lambda e^{t(A-\lambda I)}\int_0^te^{-u(A-\lambda I)}RK(u)du.
\end{align*}
Taking derivatives, and using that $K'(t) = (A-\lambda I + \lambda R)K(t)$ we obtain
\begin{align*}
D'(t) &= (A-\lambda I + \lambda R)K(t) - (A-\lambda I) e^{(A-\lambda I)t} -\lambda RK(t) \\
&\phantom{=} - (A-\lambda I) \lambda  e^{t(A-\lambda I)}\int_0^te^{-u(A-\lambda I)}RK(u)du \\
&= (A-\lambda I) D(t).
\end{align*}
Together with the initial condition $D(0) = 0$, this ODE has the unique solution $D \equiv 0$, showing that \eqref{eq:exponential_kernel} solves \eqref{eq:renewal equation}.
\end{proof}
Several alternative proofs for the above result can be given, using for example the variation-of-constants formula \eqref{eq:nagel_variation_formula}, or the generator representation \eqref{eq:pois_generator}.

\begin{example}[Growth--collapse process with Poisson jump times and uniform cutoff]
% \label{ex:growth-collapse}
We continue Example~\ref{ex:MPUR}(i) of the growth--collapse process considered in \cite{Boxma06}. The generator of the constant-growth process
$X_t = x + c t$
is
\begin{align*}
	\cA f(x) = c\,f'(x)
\end{align*}
which maps $\cP_n$ to $\cP_n$ for any $n \in \NN_0$. Hence, $X$ is a polynomial process, and its generator matrix expressed in the basis $H_n(x) = (1,x,\dots,x^n)^\top$ is
\begin{align*}
	(A_n)_{k-1,k} = c\,k,\quad k=1,\dots,n,
\end{align*}
with all other entries equal to $0$. Let $\widetilde X=(\widetilde X_t)_{t\geq0}$ denote the associated PUAR with affine reset mechanism $(r,\cZ, \mu)$ where
\begin{align*}
	r(x,z) = z x,\quad x\in S,
\end{align*}
$\cZ = [0,1]$ and $\mu$ is the uniform distribution on $[0,1]$. Then the associated reset matrix is
\begin{align*}
	R_n = \diag(1,u_1,\dots,u_n), \quad u_k \coloneqq  \int_0^1 z^k\,dz = \frac{1}{k+1},\quad k=1,\dots,n,
\end{align*}
where the $u_k$ are the moments of the uniform distribution on $[0,1]$.
%since in the one-dimensional case $D^*(k)=1$ for all $k$.
We consider Poissonian resetting with rate $\lambda>0$. Then, by \Cref{cor:poisson_renewal_times}, for $p\in P_n$ and $x\in S$, %kernel of the PPLSR $\widetilde X$ is
%\begin{align*}
%  K_n(t) = \exp(t\,(A_n+\lambda(R_n-I))),\quad t\ge0.
%\end{align*}
%Hence the moment formula~ gives,
\begin{equation}
		% \label{eq:moment-gc}
	\bE_x[p(\widetilde X_t)] = H_n(x)^\top\,e^{tB_n}\,\vec p,\quad t\geq0,
\end{equation}
where $B_n=A_n+\lambda(R_n-I)$ with
\begin{align*}
	B_n = \begin{bmatrix} 0 & c & 0 & \cdots & 0 \\
	0 & -\frac{\lambda }{2} & 2c & \cdots & 0 \\
	0 & 0 & -\frac{2\lambda}{3} & \ddots & \vdots \\
	\vdots & \vdots & \ddots & \ddots & n c \\
	0 & 0 & \cdots & 0 & -\frac{\lambda n}{n+1} \end{bmatrix}.
\end{align*}
The moments $m_k(t) = \bE_x[\widetilde X^k]$ of $\widetilde X$ (with initial value $\widetilde X_0 = x$) can now be computed from
\begin{align*}
	%= \bE_x\bigl[h_k(\widetilde X_t)\bigr]
	m_k(t)  = \bE_x[\widetilde X_t^k] = H_n(x)^\top e^{t B_n}\, e_{k},\quad t\ge0.
\end{align*}
Differentiating this formula leads to a recursive scheme of computing these moments by
\begin{equation*}
	m_k(t) = e^{-\frac{\lambda k}{k+1}t}\,x^k + c\,k\int_0^t e^{-\frac{\lambda k}{k+1}(t-s)} m_{k-1}(s)\,ds, \quad 1\le k\le n,
\end{equation*}
starting with $m_0(t) = 1$.\\
Choosing $x=0$ and $c=1$, we obtain
\begin{align*}
m_1(t) &= \frac{2}{\lambda}\left(1 - e^{-\lambda t/2}\right)\\
m_2(t) &= \frac{3}{2\lambda^2}\left(4 - 16 e^{-\lambda t/2} + 12e^{-2\lambda t/3}\right)\\
\intertext{and thus}
\var(\widetilde X) &= \frac{1}{\lambda^2}\left(2 - 16 e^{-\lambda t/2} + 18e^{-2\lambda t/3} - 4e^{-\lambda t}\right),
\end{align*}
consistent with \cite[pp.~232-235]{Boxma06}, where the same values are obtained by different methods, taking up four pages of calculations.
\end{example}
\begin{example}[Bivariate Brownian motion with total reset]
% \label{ex:bm-2d-total-reset}
Let $(B_1(t), B_2(t))_{t \geq 0}$ be a bivariate standard Brownian motion, and write $X_t = (x + B_1(t), y + B_2(t))$ for the Brownian motion starting from $(x,y) \in \RR^2$. Then $X$ is a polynomial process with generator
\begin{equation}\label{eq:generator_BM}
	\cA f(x,y)=\frac12 \left(\frac{\partial^2 f}{\partial x^2} + \frac{\partial^2 f}{\partial y^2} \right)
\end{equation}
and we set $A_n\coloneqq \cA|_{\cP_n}$ for $n\in\bN$ with respect to a graded monomial basis $H_n$. By straightforward calculations, we obtain that
\[\bE_{(x,y)}[X_t] = \begin{pmatrix}x\\y\end{pmatrix}, \qquad \cov_{(x,y)}(X_t) = \begin{pmatrix}t & 0 \\ 0 & t \end{pmatrix}.\]
We want to determine how this moment structure changes, when we add total resetting ($r \equiv 0$) at Poissonian reset times. The corresponding reset matrix is $R = e_0 e_0^\top$, where $e_0 = (1,0,0,\dotsc)^\top$. By \Cref{cor:poisson_renewal_times}, for $p\in \cP$ and $x\in S$,
\begin{equation}
	\label{eq:moment-bm}
	\bE_{(x,y)}[p(\widetilde X_t)] = H(x,y)^\top\,e^{t\,B}\,\vec p, \quad t\ge0,
\end{equation}
where $B = A + \lambda (e_0 e_0^\top - I)$. For degree $n=2$ we fix the basis $H_2(x,y) = (1,x,y,x^2, xy, y^2)^\top$. Then the generator matrix $A = A_2$ is easily calculated from \eqref{eq:generator_BM} has only two non-zero entries with
\begin{align*}
	A_{03} = A_{05} = 1, \quad A_{ij} = 0 \ \text{for all other }i\text{ and }j.
\end{align*}
Consequently, the matrix $B = A +\lambda(e_0 e_0\top-I)$ is given as
\begin{align*}
	B=\begin{bmatrix} 0 & 0 & 0 & 1 & 0 & 1 \\
	0 & -\lambda & 0 & 0 & 0 & 0 \\
	0 & 0 & -\lambda & 0 & 0 & 0 \\
	0 & 0 & 0 & -\lambda & 0 & 0 \\
	0 & 0 & 0 & 0 & -\lambda & 0 \\
	0 & 0 & 0 & 0 & 0 & -\lambda \end{bmatrix}.
\end{align*}
Using the moment formula \eqref{eq:moment-bm} and evaluating the matrix exponential we obtain the raw moments of $\widetilde X$ as
\begin{multline*}
H(x,y)^\top e^{Bt} \\
= \left(1, xe^{-\lambda t}, ye^{-\lambda t}, x^2 e^{-\lambda t} + \frac{1}{\lambda}(1 - e^{-\lambda t}), xy e^{-\lambda t}, y^2 e^{-\lambda t} + \frac{1}{\lambda}(1 - e^{-\lambda t})\right).
\end{multline*}
From this, we have
\begin{align*}
	%  \bE_x\bigl[\widetilde X_t^{(i)}\bigr]
	%    &= x_i\,e^{-\lambda t},\quad i=1,2,\\
	\bE_{(x,y)}[\widetilde X_t] &= \begin{pmatrix}x \\y \end{pmatrix} e^{-\lambda t}\quad \text{and} \\
    \cov_{(x,y)}(\widetilde X_t) &= e^{-\lambda t}(1 - e^{-\lambda t})\begin{pmatrix}x^2 & xy \\ xy & y^2 \end{pmatrix} + \frac{1 - e^{-\lambda t}}{\lambda} \begin{pmatrix} 1& 0 \\ 0 & 1 \end{pmatrix}
\end{align*}
for the bivariate Brownian motion with resetting to zero. In particular, the components of $\widetilde X$ are no longer uncorrelated when $(x,y) \neq 0$.
\end{example}

\subsection{The cross-moment formula for PUARs}
We derive a formula for the cross-moments of a polynomial process under affine resetting, i.e., for joint moments at two different times $t$ and $t+h$. Recall that for a given order $n \in \NN_0$, $H_n$ represents a graded monomial basis of $\cP_n$, $A_n$ the generator matrix of the Markov process $X$, and $R_n$ the reset matrix. Moreover, $M_q$ is the multiplication matrix associated to the polynomial $q$, see \Cref{sec:polynomial_notation}.

\begin{thm}[Cross-moment formula]\label{cor:cross moment 2}
Let $k,n \in \NN_0$. For all $p\in\cP_n$, $q \in \cP_k$ and $x\in S$,
\begin{align*}
	\bE_x\!\left[ q(\widetilde X_t)p(\widetilde X_{t+h}) \right] = H_{n+k}(x)^\top C_{q,h}(t)\vec p,\quad t,h\geq0,
\end{align*}
where $C_{q,h}\colon\bR_+\to\bR^{D(n+k)\times D(n)}$ is the unique locally bounded kernel satisfying
\begin{align}
	\label{eq:correct-cross-moment-renewal}
	C_{q,h}(t) = G_{q,h}(t) + \int_0^t e^{uA_{n+k}}R_{n+k}C_{q,h}(t-u)\,F(du),
\end{align}
for every $h\ge0$, where
\begin{align}
\label{eq:cross-source-fixed-lag}
G_{q,h}(t)
&=
e^{tA_{n+k}}M_qe^{hA_n}\overline F(t+h) \\\nonumber
&\quad+
e^{tA_{n+k}}M_q
\int_{(t,t+h]}
    e^{(u-t)A_n}R_nK_n(t+h-u)\,F(du).
\end{align}
and $K_n(t)$ is the kernel from \Cref{thm:moment formula 2}.
\end{thm}

\begin{proof}
Fix $p\in\cP_n$, $q \in \cP_k$, $x\in S$, and $h\geq0$. Set $L_{n+k}(du)\coloneqq e^{uA_{n+k}}R_{n+k}F(du)$ for $u\geq0$ and
\begin{align*}
	C_{q,h}^{(0)}(t) &\coloneqq G_{q,h}(t),\\
	C_{q,h}^{(m)}(t) &\coloneqq (L_{n+k}*C_{q,h}^{(m-1)})(t)= \int_0^t e^{uA_{n+k}}R_{n+k} C_{q,h}^{(m-1)}(t-u)\,F(du),
\end{align*}
for each $m\in\bN$. Analogous to the proof of \Cref{thm:moment formula 2} we claim that
\begin{align}
	\label{eq:cross-m-reset-contribution}
	\bE_x\!\left[ q(\widetilde X_t)p(\widetilde X_{t+h}) \mathds 1_{\{N_t=m\}} \right] = H_{n+k}(x)^\top C_{q,h}^{(m)}(t)\vec p, \qquad m\in\bN_0.
\end{align}
For $m=0$, the event $\{N_t=0\}$ is, up to a null set, the disjoint union
\begin{align*}
	\{\tau_1>t+h\} \cup \{t<\tau_1\leq t+h\}.
\end{align*}
On $\{\tau_1>t+h\}$, by independence of $\tau_1$ and $X$, the conditional moment formula \eqref{eq:polynomial_moment_classic} for the polynomial process $X$, and the multiplication identity \eqref{eq:matrix_Mq} we have
\begin{align}
	\bE_x\!\left[ q(\widetilde X_t)p(\widetilde X_{t+h}) \mathds 1_{\{\tau_1>t+h\}} \right] &= \bE_x[q(X_t)p(X_{t+h})] \overline F(t+h) \nonumber\\
	&= \bE_x\!\left[ q(X_t) \bE_x[p(X_{t+h})\mid X_t] \right] \overline F(t+h) \nonumber\\
	&=  \bE_x\!\left[ q(X_t)H_n(X_t)^\top e^{hA_n} \right] \vec p\,\overline F(t+h)\nonumber\\
		% \label{eq:cross-no-reset-until-second-time}
	&= H_{n+k}(x)^\top e^{tA_{n+k}}M_qe^{hA_n}\vec p \, \overline F(t+h).
\end{align}
On $\{t<\tau_1\leq t+h\}$, applying \Cref{thm:moment formula 2}, \Cref{lem:conditional_reset}, and \Cref{thm:moment_formula_classic} gives
\begin{align*}
	&\bE_x\!\left[ q(\widetilde X_t)p(\widetilde X_{t+h}) \mathds 1_{\{t<\tau_1\leq t+h\}} \right] \nonumber\\
	%\label{eq:cross-first-reset-after-s}
    &\quad= \bE_x\!\left[ q(\widetilde X_t) \bE_x\!\left[p(\widetilde X_{t+h}) \mid \cG_{\tau_1}\right] \mathds 1_{\{t<\tau_1\leq t+h\}} \right] \nonumber \\
    &\quad =\bE_x\!\left[ q(\widetilde X_t) H(\widetilde X_{\tau_1})^\top K_n(t+h-\tau_1) \mathds 1_{\{t<\tau_1\leq t+h\}} \right] \nonumber \\
	&\quad= \bE_x\!\left[ q(X_t) H(X_{\tau_1-})^\top R_n \,K_n(t+h-\tau_1) \mathds 1_{\{t<\tau_1\leq t+h\}} \right] \nonumber \\
    &\quad= \int_{(t,t+h]} \bE_x\!\left[ q(X_t) H(X_{u})^\top R_n\,K_n(t+h-u) \right] F(du) \nonumber \\
    &\quad= \int_{(t,t+h]} \bE_x\!\left[ q(X_t) H(X_{t})^\top e^{(u-t)A_n} R_n\,K_n(t+h-u) \right] F(du) \nonumber \\
    &\quad= H_{n+k}(x)^\top e^{tA_{n+k}} M_q \int_{(t,t+h]} e^{(u-t)A_n} R_n\,K_n(t+h-u) F(du). \nonumber 
    %&\quad= H_{n+k}(x)^\top e^{tA_{n+k}} M_q \int_{(t,t+h]}  e^{uA_n} R_n\,K_n(h-u) dF(t+u). \nonumber \\
\end{align*}
Adding the last two equations together, gives
\begin{align*}
	\bE_x\!\left[ q(\widetilde X_t)p(\widetilde X_{t+h}) \mathds 1_{\{N_t=0\}} \right] = H_{n+k}(x)^\top G_{q,h}(t)\vec p.
\end{align*}
Thus \eqref{eq:cross-m-reset-contribution} holds for $m=0$ and we proceed by induction. Let $m\in\bN$ and suppose that \eqref{eq:cross-m-reset-contribution} holds for $m-1$. Since $N_t=m$ implies $\tau_1\le t$, all subsequent expectations can be understood with an additional factor
$\mathds 1_{\{\tau_1\le t\}}$ inside them. Then, conditioning on the first reset time, using the induction hypothesis, and applying the moment formula, gives
\begin{align*}
	&\bE_x\!\left[q(\widetilde X_t)p(\widetilde X_{t+h})\mathds 1_{\{N_t=m\}}\right]\\
	&\quad = \bE_x\!\left[\bE_x\!\left[q(\widetilde X_t)p(\widetilde X_{t+h})  \mathds 1_{\{N_t=m\}} \mid \cG_{\tau_1} \right]\right]\\
    &\quad = \bE_x\!\left[\bE_{\widetilde X_{\tau_1}}\!\left[q(\widetilde X_{t-\tau_1})p(\widetilde X_{t+h-\tau_1}) \mathds 1_{\{N_{t-\tau_1}=m-1\}}\right]\right]\\
&\quad = \bE_x\! \left[H_{n+k}(\widetilde X_{\tau_1})^\top C_{q,h}^{(m-1)}(t-\tau_1)\right] \vec p\\
&\quad = \bE_x\! \left[H_{n+k}(X_{\tau_1-})^\top R_{n+k} C_{q,h}^{(m-1)}(t-\tau_1)\right] \vec p\\
&\quad = H_{n+k}(x)^\top \bE_x\! \left[e^{\tau_1 A_{n+k}} R_{n+k} C_{q,h}^{(m-1)}(t-\tau_1)\right] \vec p\\
&\quad = H_{n+k}(x)^\top \int_0^t e^{u A_{n+k}} R_{n+k} C_{q,h}^{(m-1)}(t-u) F(du) \,\vec p \\
&\quad = H_{n+k}(x)^\top C_{q,h}^{(m)}(t) \vec p.
\end{align*}
Hence \eqref{eq:cross-m-reset-contribution} holds for every $m\in\bN_0$. The same estimates as in the proof of \Cref{thm:moment formula 2}, with $G$, $L$, and $K^{(m)}$ replaced by $G_{q,h}$, $L_{n+k}$, and $C_{q,h}^{(m)}$, respectively, show that
\begin{align*}
	C_{q,h}(t) \coloneqq \sum_{m=0}^{\infty}C_{q,h}^{(m)}(t)
\end{align*}
is finite and locally bounded. Summing \eqref{eq:cross-m-reset-contribution} over $m\in\bN_0$ therefore gives
\begin{align*}
	\bE_x\!\left[ q(\widetilde X_t)p(\widetilde X_{t+h}) \right] = H_{n+k}(x)^\top C_{q,h}(t)\vec p.
\end{align*}
Moreover,
\begin{align*}
	C_{q,h}(t) &= G_{q,h}(t) + \sum_{m=1}^{\infty} (L_{n+k}*C_{q,h}^{(m-1)})(t) \nonumber\\
	&= G_{q,h}(t) + (L_{n+k}*C_{q,h})(t),
\end{align*}
which is precisely \eqref{eq:correct-cross-moment-renewal}. Finally, uniqueness among locally bounded solutions follows by the same argument as in the proof of \Cref{thm:moment formula 2}, applied to
\begin{align*}
	C_{q,h}=G_{q,h}+L_{n+k}*C_{q,h}.\qedhere
\end{align*}

%Setting $h=t-s$ proves \eqref{eq:correct-cross-moment-formula}.
\end{proof}
\begin{cor}[Poissonian resetting]
	In the setting of \Cref{cor:cross moment 2}, assume Poissonian resetting at rate $\lambda > 0$, i.e., $F(t)=1- e^{-\lambda t}$. Then, for all $p\in\cP_n$ , $q \in \cP_k$ and $x\in S$,
	\begin{align}\label{eq:poissonian_cross}
		\bE_x
[q(\widetilde X_t)p(\widetilde X_{t+h})]
		=H_{n+k}(x)^\top e^{tB_{n+k}}M_qe^{hB_n}\vec p,\quad t,h\geq 0,
	\end{align}
where $B_j\coloneqq A_j+\lambda(R_j-I_{D(j)})$ for $j\in\{n,n+k\}$.
%	\begin{align*}
%		B_k\coloneqq A_k+\lambda(R_k-I_{D(k)}),\quad k\in\{n,n+k\}.
%	\end{align*}
\end{cor}
\begin{proof}
For square matrices $U$ and $V$, the variation-of-constants formula \cite[Corollary~III.1.7, p.~161]{EngelNagel00} gives
\begin{align}\label{eq:nagel_variation_formula}
	e^{t(U+V)}
	%=e^{tU}+\int_0^t e^{(t-s)U}Ve^{s(U+V)}\,ds
	=e^{tU}+\int_0^t e^{uU}Ve^{(t-u)(U+V)}\,du.
\end{align}
%Set
%\begin{align*}
%	U_k\coloneqq A_k-\lambda I_{D(k)},\qquad
%	B_k=U_k+\lambda R_k,\qquad k\in\{n,n+k\}.
%\end{align*}
Set $U_j=A_j-\lambda I_{D(j)}$ for $j\in\{n,n+k\}$. Applying \eqref{eq:nagel_variation_formula} with $U=U_n$ and $V=\lambda R_n$ yields
\begin{align}
	\label{eq:duhamel-Bn}
	e^{hB_n}
	=e^{hU_n}+\int_0^h\lambda e^{uU_n}R_ne^{(h-u)B_n}\,du,
\end{align}
where $B_n=A_n+\lambda(R_n-I_{D(n)})$. By \Cref{cor:poisson_renewal_times}, we obtain from \eqref{eq:cross-source-fixed-lag}
\begin{align*}
	G_{q,h}(t)
	&=e^{-\lambda(t+h)}e^{tA_{n+k}}M_qe^{hA_n}\\
	&\quad+\int_t^{t+h}\lambda e^{-\lambda u}
	e^{tA_{n+k}}M_qe^{(u-t)A_n}
	R_ne^{(t+h-u)B_n}\,du\\
	&=e^{t(A_{n+k}-\lambda I_{D(n+k)})}
	M_qe^{h(A_n-\lambda I_{D(n)})}\\
	&\quad+e^{t(A_{n+k}-\lambda I_{D(n+k)})}M_q
	\int_t^{t+h}\lambda e^{(u-t)(A_n-\lambda I_{D(n)})}
	R_ne^{(t+h-u)B_n}\,du\\
	&=e^{tU_{n+k}}M_qe^{hU_n}
	+e^{tU_{n+k}}M_q
	\int_0^h\lambda e^{vU_n}R_ne^{(h-v)B_n}\,dv\\
	&=e^{tU_{n+k}}M_q
	\left(e^{hU_n}
	+\int_0^h\lambda e^{vU_n}R_ne^{(h-v)B_n}\,dv\right)\\
	&=e^{tU_{n+k}}M_qe^{hB_n},
\end{align*}
where we used the substitution $v=u-t$ and \eqref{eq:duhamel-Bn}. Consequently, \eqref{eq:correct-cross-moment-renewal} becomes
\begin{align*}
	C_{q,h}(t)
	&=e^{tU_{n+k}}M_qe^{hB_n}
	+\int_0^t\lambda e^{uU_{n+k}}R_{n+k}C_{q,h}(t-u)\,du.
\end{align*}
Applying \eqref{eq:nagel_variation_formula} once more with $U=U_{n+k}$ and $V=\lambda R_{n+k}$ gives
\begin{align*}
	e^{tB_{n+k}}M_qe^{hB_n}
	&=e^{tU_{n+k}}M_qe^{hB_n}\\
	&\quad+\int_0^t\lambda e^{uU_{n+k}}R_{n+k}e^{(t-u)B_{n+k}}M_qe^{hB_n}\,du,
\end{align*}
Hence, uniqueness in \Cref{cor:cross moment 2} yields \eqref{eq:poissonian_cross}.

\end{proof}

\section{Stochastic representations of PUARs}\label{sec:stochastic_res}
We provide two stochastic representations of the PUAR $\widetilde X=(\widetilde X_t)_{t\geq0}$, extending the stochastic representations of polynomial processes without resetting from \Cref{prop:stochastic}. We define
\begin{equation}\label{eq:M}
\mathsf{M}(\omega; dt, dz) = \sum_{i=1}^\infty \delta_{(\tau_i(\omega), Z_i(\omega))}(dt,dz)
\end{equation}
which is a random counting measure on $\cB(\Rplus) \otimes \cB(\cZ)$, related to the marks $Z_i$ of $\widetilde X$. Furthermore, we define
\begin{equation}\label{eq:elapsed_time_V}
V_t = t - \tau_{N_t},
\end{equation}
which is the elapsed time since the last reset. The main results of this section are the following proposition and the theorem on stochastic representations, both of which will be shown after a series of auxilliary lemmas below:
\begin{prop}Let $\widetilde X$ be a PUAR with affine reset mechanism $(r,\cZ,\mu)$ and assume that the inter-arrival distribution $F$ has a density $f$, such that the hazard rate $h(t)=\frac{f(t)}{1 - F(t)}$ is well defined for all $t\geq0$. \label{lem:M_and_V}
Then, the elapsed time $V$ satisfies
\begin{equation}
	\label{eq:time since reset}
dV_t =dt-V_{t-}\int_{\cZ}\mathsf{M}(dt,dz),
\end{equation}
and the predictable compensator of $\mathsf{M}$ is
\begin{equation}% \label{eq:compensator_M}
	\nu^M(dt,dz) = h(V_{t-})\,\mu(dz)\,dt.
\end{equation}

\end{prop}

% and that $\bG=(\cG_t)_{t\geq0}$ denotes the progressive enlargement of $\bF$.
\begin{thm}[Stochastic Representation of PUARs]\label{prop:repesent_reset}
Let $\widetilde X$ be a PUAR with affine reset mechanism $(r,\cZ,\mu)$, and assume that the inter-arrival distribution $F$ has a density $f$, such that the hazard rate $h(t)=\frac{f(t)}{1 - F(t)}$ is well defined for all $t\geq0$.
%Further, let $(Z_i)_{i\in\bN}$ be an i.i.d. randomisation sequence with $Z_1\sim\mu$ and let $\bG=(\cG_t)_{t\geq0}$ denote the progressive enlargement of $\bF$ with respect to $(\bT,(Z_i)_{i\in\bN})$.
%\begin{align*}
%		\gamma_t(z)\coloneqq r(\widetilde X_{t^-},z)-\widetilde X_{t^-}.
%	\end{align*}
	\begin{enumerate}[(i)]
	\item Let $X$ be a polynomial diffusion as in \Cref{prop:stochastic}(i). Then $\widetilde X$ satisfies
\begin{align}
	\label{eq:dyn reset X}
	d\widetilde X_t =\mathfrak{b}(\widetilde X_t)\,dt+\sigma(\widetilde X_t)\,dW_t+\int_{\cZ}\left(r(\widetilde X_{t^-},z)-\widetilde X_{t^-}\right)\,\mathsf{M}(dt,dz),
\end{align}
where $\mathsf{M}(dt,dz)$ is the random measure from \eqref{eq:M}.
\item Let $X$ be a polynomial process with semimartingale representation as in \Cref{prop:stochastic}(ii).
	%and let $\widetilde X$ denote the PPLSR. Assume that $\mu$ has finite absolute first and second moments.
%	in particular,\todo{remark that condition is satisfied by all in this paper considered resets, simplify assumption.}
%	\begin{align}\label{eq:nu moment}
%		\int_{\cZ}\|z-I\|\,\mu(dz)<\infty.
%	\end{align}
	Then, on any of the filtered probability spaces $(\Omega, \cF, \bG, \bP_x)$, for $x \in S$, the process $\widetilde X$ is a special semimartingale with characteristics $(\widetilde B, \widetilde C, \widetilde\nu)$ associated to the 'truncation function' $\chi(\xi) =\xi$ where $\widetilde C=C$ satisfies \eqref{eq:C_eta} and
\begin{align}
	\label{eq:B_reset}
	&\widetilde B_{t,i}=\int_0^{t} \left(\mathfrak{b}_i(\widetilde X_{s})+ h(V_{s^-})\int_{\cZ}\gamma_s^i(z)\,\mu(dz)\right)\,ds,\\
	%&\widetilde C_{t,ij}=C_{t,ij}
	%=\int_0^t c_{s,ij}\,ds,
	\label{eq:C_res}
	&\widetilde C_{t,ij}+\int_0^t\int_{\bR^d}\xi_i\xi_j \,\widetilde\nu(ds, d\xi)\\
	\nonumber &\hspace{10mm}= \int_0^{t} \left(\mathfrak{a}_{ij}(\widetilde X_{s}) +h(V_{s-})\int_{\cZ}\gamma^i_s(z)\gamma^j_s(z)\,\mu(dz)\right)\,ds,
\end{align}
where $\gamma_t(z)=r(\widetilde X_{t^-},z)-\widetilde X_{t^-}$, $\mathfrak{b}_i \in \cP_1$, and $\mathfrak{a}_{ij} \in \cP_2$
%\todo{add representation for $C$ (in remark)? finite second moments needed (in cuhiero we get that implicitly by ito's formula, change notation for indicator set.}
for all $i,j \in \set{1, \dotsc, d}$. Moreover, $\widetilde \nu$
%is the sum of $\nu$ given in \eqref{eq:C_eta} and the predictable compensator of the reset jumps and
can be written as
\begin{align}
	\label{eq:nu_reset}
	\widetilde\nu(dt,d\xi) =\widetilde K_t(d\xi)\,dt,\quad \widetilde K_t(d\xi) =K_{t}(d\xi) +h(V_{t^-})\int_{\cZ}
\mathds 1_{\{\gamma_t(z)\neq0\}}
\delta_{\gamma_t(z)}(d\xi)\,\mu(dz),
\end{align}
where $K_{t}(d\xi)$ denotes the kernel from \eqref{eq:K} and $V_t$ satisfies \eqref{eq:time since reset}.
\end{enumerate}
\end{thm}

We state two auxiliary lemmata in order to identify the compensator of the random measure appearing in \eqref{eq:dyn reset X}. Recall that $(\tau_i)_{i \in \NN}$ is an a.s. increasing sequence of random times and $\tau_{i+1} - \tau_i$, $i\in\bN$, are i.i.d. with hazard rate $h$.
\begin{lem}\label{lem:single jump}
%	\todo{encode in the filtration $\cG_t$.}Let $S$ be a non-negative random variable with hazard rate $h$.
The counting process $N_t^{(1)}=\mathds 1_{\{\tau_1\le t\}}$ has predictable compensator
\begin{align*}
	\Lambda_t^{(1)} = \int_0^t \mathds 1_{\{\tau_1>s\}}h(s)\,ds.
\end{align*}
\end{lem}
\begin{proof}
	From $\overline F(t)=\exp\left(-\int_0^t h(s)\,ds\right)$ \cite[Section 1.1.2]{Aalen08} we have, for any $s<t$,
\begin{align*}
	\mathbb{P}(\tau_1>t \mid \tau_1>s) = \frac{\overline F(t)}{\overline F(s)} = \exp\left(-\int_s^t h(u)\,du\right).
\end{align*}
The conditional expectation of the jump indicator over $(s,t]$ is therefore
\begin{align*}
	%&=
	%\mathbb{E}\!\left[
	%\int_s^t \mathds{1}_{\{S>u\}}h(u)\,du
	%\,\middle|\, \cG_s
	%],
	\mathbb{E}[ \mathds{1}_{\{\tau_1\le t\}}-\mathds{1}_{\{\tau_1\le s\}}\mid \cG_s ] = \mathds{1}_{\{\tau_1>s\}}\left(1-\exp\left(-\int_s^t h(u)\,du\right)\right),
\end{align*}
%where the last equality uses Fubini and the explicit conditional survival function.
By Fubini and since $\bE[\mathds{1}_{\{\tau_1>u\}}\mid\cG_s]=\mathds{1}_{\{\tau_1>s\}}\bP(\tau_1>u\mid \tau_1>s)$ for $u>s$ we get
\begin{align*}
	%\bE\left[\Lambda^S_t-\Lambda^S_s]
	\mathbb{E}\left[\int_s^t \mathds{1}_{\{\tau_1>u\}}h(u)\,du\,\middle|\, \cG_s\right] &=\int_s^t\bE[\mathds{1}_{\{\tau_1>u\}}\mid\cG_s]h(u)\,du\\
	%&=\int_s^t\bP(S>u\mid S>s)h(u)\,du\\
	&=\mathds{1}_{\{\tau_1>s\}}\int_s^t\exp\left(-\int_s^u h(v)\,dv\right)h(u)\,du,
\end{align*}
Then the substitution $H(u)=\int_s^uh(v)\,dv$ yields
\begin{align*}
	\mathbb{E}\left[\int_s^t \mathds{1}_{\{\tau_1>u\}}h(u)\,du\,\middle|\, \cG_s\right] &=\mathds{1}_{\{\tau_1>s\}}\int_s^t\exp\left(-H(u)\right)H'(u)\,du \\
	&=\mathds{1}_{\{\tau_1>s\}}\left(1-\exp\left(-\int_s^t h(u)\,du\right)\right),
\end{align*}
and thus we have shown that $\mathbb{E}[N_t^{(1)}-N_s^{(1)}\mid\cG_s]=\mathbb{E}[\Lambda_t^{(1)}-\Lambda_s^{(1)}\mid\cG_s]$ for {$s<t$}. Since $N_s^{(1)}-\Lambda_s^{(1)}$ is $\cG_s$-measurable, adding this term to both sides yields the desired martingale property of $N^{(1)}-\Lambda^{(1)}$ with respect to $\bG$. Moreover, $\Lambda^{(1)}$ is left-continuous and therefore predictable, which concludes the proof.
\end{proof}

Similar results can be also found in Theorem 1.3.1 of \cite{fleming1991counting}.

\begin{lem}\label{lem:unmarked jumps}
%	Let $(\tau_n)_{n\in\bN}$ be an a.s. increasing sequence of random times and let $(S_n)_{n\in\bN}$ be i.i.d. positive random variables with hazard rate $h$ such that
%\begin{align}\label{eq:renewal times}
%\tau_0=0,
%\quad
%\tau_n=\sum_{k=1}^n S_k, \quad n\in\bN.
%%\quad
%%N_t=\sum_{n\ge 1}\mathds 1_{\{T_n\le t\}}.
%\end{align}
The process $N_t=\sum_{n =1}^\infty \mathds 1_{\{\tau_n\le t\}}$ has predictable compensator
\begin{align*}
	\Lambda_t = \int_0^t h(V_{s-})\,ds,
\end{align*}
where $V_t=t-\tau_{N_t}$ is the elapsed time since the last jump of $N_t$, and $V_{t-}=t-\tau_{N_{t-}}$.
\end{lem}
\begin{proof}
	For each $n\in\bN$ write $N_t^{(n)}=\mathds{1}_{\{\tau_n\le t\}}$. Then
\begin{align*}
	N_t=\sum_{n\ge1}N_t^{(n)}.
\end{align*}
Since the inter-arrival times are strictly positive a.s., the processes $N^{(n)}$ have disjoint jump times. Hence the compensator of $N$ is the sum of the compensators of the $N^{(n)}$. Fix $n\ge1$. Since $\tau_n-\tau_{n-1}$ is independent of $\cG_{\tau_{n-1}}$ and has hazard rate $h$, the process $N^{(n)}$ is, after the time shift $\tau_{n-1}$, equivalent in law to the single-jump process treated in \Cref{lem:single jump}. Therefore its predictable compensator is
\begin{align*}
	\Lambda_t^{(n)} = \int_0^t \mathds{1}_{\{\tau_{n-1}<s\le \tau_n\}} h(s-\tau_{n-1})\,ds .
\end{align*}
Summing over $n$ gives
\begin{align*}
	\Lambda_t = \sum_{n\ge1}\Lambda_t^{(n)} = \int_0^t \sum_{n\ge1} \mathds{1}_{\{\tau_{n-1}<s\le \tau_n\}} h(s-\tau_{n-1})\,ds .
\end{align*}
For each $s$, we find exactly one $n$ satisfying $\tau_{n-1}<s\le \tau_n$, namely $n=N_{s-}+1$. On this set, we have $s-\tau_{n-1}=s-\tau_{N_{s-}}=V_{s-}$. Thus
\begin{align*}
	\Lambda_t=\int_0^t h(V_{s-})\,ds,
\end{align*}
and hence $N-\Lambda$ is a $\mathbb G$-martingale.
%\begin{align*}
%	\bE[N_t^S-\Lambda^S_t\mid\cG_s]
%	=\bE[N_t^S-N^S_s+N^S_s-\Lambda^S_t\mid\cG_s]
%	=N^S_s-\Lambda^S_s
%\end{align*}
\end{proof}

\begin{proof}[Proof of \Cref{lem:M_and_V}]In order to show that the integer-valued random measure $\mathsf{M}(dt,dz)$
%\begin{align*}
%\mathsf{M}(dt,dz)
%=
%\sum_{n\ge 1}\delta_{Z_n}(dz)\delta_{T_n}(dt)
%\end{align*}
has predictable compensator $\nu^M(dt,dz)=h(V_{t-})\,\mu(dz)\,dt$ we must show by \cite[Theorem II.1.8]{Jacod03} that for every non-negative $\mathbb G$-predictable process $\varphi(t,z,\omega)$
\begin{align}
	%\mathbb E\left[\int_0^\infty \int_{\cZ}\varphi(s,z)\,h(V_{s-})\,\mu(dz)\,ds\right]
	\label{eq:compensator_cond}
	\mathbb E\left[\int_0^\infty \int_{\cZ}\varphi(s,z)\,\mathsf{M}(ds,dz)\right] =\mathbb E\left[\int_0^\infty \int_{\cZ}\varphi(s,z)\,h(V_{s-})\,\mu(dz)\,ds\right]
\end{align}
holds true. First, assume $\varphi$ is bounded. Then, by Fubini and \eqref{eq:M}, the left-hand side equals
\begin{align*}
	\mathbb E\left[\sum_{n\ge 1}\varphi(\tau_n,Z_n)\right] &=\mathbb E\left[\sum_{n\ge 1}\mathbb E[\varphi(\tau_n,Z_n)\mid \cG_{\tau_n-}]\right],
\end{align*}
and the fact that $\varphi$ is predictable and $Z_n$ is independent from $\cG_{\tau_n-}$, yields
\begin{align*}
	\mathbb E[ \varphi(\tau_n,Z_n)\mid \cG_{\tau_n-} ] = \int_{\cZ}\varphi(\tau_n,z)\,\mu(dz) \eqqcolon  \psi(\tau_n),
\end{align*}
where $\psi(t)$ defines a $\mathbb G$-predictable process. Therefore
\begin{align*}
	\mathbb E\left[\sum_{n\ge 1}\varphi(\tau_n,Z_n)\right] =\mathbb E\left[\sum_{n\ge 1}\psi(\tau_n)\right] =\mathbb E\left[\int_0^\infty \psi(s)\,dN_s\right],
\end{align*}
where $N=(N_t)_{t\geq0}$ is a counting process defined as in \Cref{lem:unmarked jumps}. Then, by \Cref{lem:unmarked jumps}, the compensator of $N$ is $h(V_{t-})\,dt$. Thus, by \cite[Theorem II.1.8]{Jacod03}, the last expectation equals
\begin{align*}
	\mathbb E\left[ \int_0^\infty \psi(s)\,h(V_{s-})\,ds \right] = \mathbb E\left[ \int_0^\infty \int_{\cZ} \varphi(s,z)\,h(V_{s-})\,\mu(dz)\,ds \right],
\end{align*}
which shows \eqref{eq:compensator_cond} for bounded $\varphi$. For general non-negative $\mathbb G$-predictable $\varphi$ the assertion follows by monotone convergence. \par
It remains to show that $V$ is governed by \eqref{eq:elapsed_time_V}. For $t\in[\tau_n,\tau_{n+1})$, one has $V_t=t-\tau_n$. At $\tau_n$,
\[
    V_{\tau_n}=0,
    \qquad
    V_{\tau_n-}=\tau_n-\tau_{n-1},
    \qquad
    \Delta V_{\tau_n}=-V_{\tau_n-}.
\]
Since $V_0=0$, for $t\geq0$ it follows that
\[
\begin{aligned}
    V_t
    &=t+\sum_{0<s\le t}\Delta V_s\\
    &=t-\sum_{n\ge1:\,\tau_n\le t}V_{\tau_n-}\\
    &=t-\int_{(0,t]}\int_{\cZ}
          V_{s-}\,\mathsf M(ds,dz).\qedhere
\end{aligned}
\]
%which equivalently gives \eqref{eq:elapsed_time_V}.
\end{proof}

\begin{proof}[Proof of \Cref{prop:repesent_reset}]
For (i) we note that on each interval $[\tau_n,\tau_{n+1})$, the process $\widetilde X$ evolves according to the dynamics
\begin{align*}
	d\widetilde X_t=b(\widetilde X_t)\,dt+\sigma(\widetilde X_t)\,dW_t.
\end{align*}
At the reset time $\tau_n$ it jumps according to $\Delta \widetilde X_{\tau_n}=r(\widetilde X_{\tau_n-},Z_n)-\widetilde X_{\tau_n-}$. The jump contribution up to time $t$ is then
\begin{align*}
	\sum_{n:\tau_n\le t}\left(r(\widetilde X_{\tau_n-},Z_n)-\widetilde X_{\tau_n-}\right) =\int_0^t \int_{\cZ}\left(r(\widetilde X_{s-},z)-\widetilde X_{s-}\right)\mathsf{M}(ds,dz).
\end{align*}
By the independence of $\mathsf{M}$ from $X$, we get
\begin{align*}
	d\widetilde X_t = b(\widetilde X_t)\,dt + \sigma(\widetilde X_t)\,dW_t + \int_{\cZ} \left( r(\widetilde X_{t-},z) - \widetilde X_{t-} \right) \mathsf{M}(dt,dz).
\end{align*}
\par

For (ii), by the Grigelionis representation theorem \cite[Theorem 2.1.2]{Jacod12}, we may, after passing to a 'very good' filtered extension \cite[(2.1.19)]{Jacod12} if necessary, represent $X$ with respect to the truncation function $\chi(\xi)=\xi$ by
\begin{align*}
	X_t =X_0+\int_0^t\mathfrak b(X_s)\,ds+\int_0^t\sigma_s\,dW_s+\int_0^t\int_{\bR^d}\xi\,\bigl(\cJ(ds,d\xi)-K_s(d\xi)\,ds\bigr),
\end{align*}
where $\sigma_t\sigma_t^\top=c_t$ satisfies \eqref{eq:C}, $\mathfrak b$ is given in \eqref{eq:B}, and $\cJ(dt,d\xi)$ is the jump measure of $X$ with predictable compensator $K_t(d\xi)$ satisfying \eqref{eq:K}. Since the reset is independent of $X$ and jump times and reset times a.s. do not coincide, we get
\begin{align}
	\label{eq:SDE reset}
	\widetilde X_t &=\widetilde X_0+\int_0^t\mathfrak b(\widetilde X_s)\,ds+\int_0^t\sigma_s\,dW_s +\int_0^t\int_{\bR^d}\xi\,\bigl(\cJ(ds,d\xi)-K_s(d\xi)\,ds\bigr)\\
	\nonumber &\hspace{5mm}+\int_0^t\int_{\cZ}\gamma_s(z)\,\mathsf{M}(ds,dz),
\end{align}
where the predictable compensator of $\mathsf{M}$ is $\nu^\mathsf{M}(dt,dz)=h(V_{t^-})\,\mu(dz)\,dt$ by \Cref{lem:M_and_V}.
%Since $\widetilde X_t=\widetilde X_{t-}$ for Lebesgue-a.e. $t$, we may use $\widetilde X_t$ in the $dt$-integrals and $\widetilde X_{t-}$ in the predictable jump integrands.
Note that $K_t(d\xi)$ now denotes the intrinsic jump kernel evaluated at the reset path of $\widetilde X$. Next, we show that
\begin{align*}
	B_t^R\coloneqq\int_0^t\int_{\cZ}\gamma_s(z)\,\nu^\mathsf{M}(ds,dz) =\int_0^t\int_{\cZ}\gamma_s(z)\,h(V_{s-})\mu(dz)\,ds
\end{align*}
is locally integrable. Recall that $r(x,z)=z_0 + \bar z \cdot x$ is affine and thus,
\begin{align*}
	\gamma_s(z)=z_0 + (\bar z-I)\cdot X_{s-}.
\end{align*}
Let $\rho_n\coloneqq n\wedge\inf\{t\ge 0:\|\widetilde X_t\|>n\}$ for $n\in\bN$. Then on $[0,\rho_n]$ we get
\begin{align*}
	&\|\gamma_s(z)\| \le\|\widetilde X_{s-}\|\,\|\bar z-I\|+\|z_0\|\le n\|z-I\|+\|x_0\|,\\
	&\|\gamma_s(z)\|^2\le 2n^2\|\bar z-I\|^2+2\|z_0\|^2.
\end{align*}
Hence, for every $T>0$, setting $k_l = \int_{\cZ} \norm{z_0}^l \mu(dz)$ and $m_l=\int_{\cZ}\|\bar z-I\|^l\,\mu(dz)$ for $l\in\{1,2\}$, we get
\begin{align}
	\label{eq:first_moment_res}
	&\int_0^{T\wedge\rho_n} h(V_{s-})\int_{\cZ}\|\gamma_s(z)\|\,\mu(dz)\,ds \le(k_1 + nm_1)\int_0^{T\wedge\rho_n}h(V_{s-})\,ds,\\
	\label{eq:second_mom_res}
	&\int_0^{T\wedge\rho_n}h(V_{s-})\int_{\cZ}\|\gamma_s(z)\|^2\,\mu(dz)\,ds \le(2k_2 + 2n^2m_2)\int_0^{T\wedge\rho_n}h(V_{s-})\,ds.
\end{align}
Now $k_1, k_2, m_1$ and $m_2$ are finite by assumption and after localizing with respect to $\int_0^t h(V_{s-})\,ds$, the right-hand sides are finite. Consequently, by \eqref{eq:first_moment_res}, $B_t^R$ is locally integrable and by \cite[Proposition I.1.28]{Jacod03} we can conclude that
\begin{align*}
	%\int_0^t\int_{\cZ}\gamma(\widetilde X_{s^-},z)\,\mathsf{M}(ds,dz)
	%&\hspace{2mm}+\int_0^th(V_{s-})\int_{\cZ}\gamma(\widetilde X_{s^-},z)\,\mu(dz)\,ds,
	L_t^R\coloneqq\int_0^t\int_{\cZ}\gamma_s(z)\mathsf{M}(ds,dz)-B^R_t
\end{align*}
is a (purely discontinuous) local martingale. Substituting this into \eqref{eq:SDE reset} yields
\begin{align*}
	\widetilde X_t = X_0+\widetilde B_t+\widetilde L_t,
\end{align*}
where $\widetilde B_t=B_t+B_t^R$ is predictable and of locally finite variation, and
\begin{align*}
	\begin{aligned} \widetilde L_t =L_t^R+\int_0^t\sigma_s\,dW_s+\int_0^t\int_{\mathbb R^d}\xi\bigl(\cJ(ds,d\xi)-K_s(d\xi)\,ds\bigr) \end{aligned}
\end{align*}
is a local martingale. Therefore $\widetilde X$ is a special semimartingale. It remains to identify $\widetilde\nu$ and $\widetilde C$. The jump measure $\cJ^R$ of $\widetilde X$ at reset times is given by
\begin{align*}
		\cJ^R(dt,d\xi)
		=\int_{\cZ}\mathds 1_{\{\gamma_t(z)\neq0\}}
		\delta_{\gamma_t(z)}(d\xi)\,\mathsf{M}(dt,dz),
	\end{align*}
	Since $\gamma_t(z)$ is predictable, its predictable compensator $\nu^R$ is the image of $\nu^M$ under $\gamma$ restricted to non-zero jump sizes, namely $\nu^R(dt,d\xi)=K^R_t(d\xi)\,dt$ where
	\begin{align*}
		%\label{eq:reset compensator}
		K^R_t(d\xi)
		=h(V_{t-})\int_{\cZ}\mathds 1_{\{\gamma_t(z)\neq0\}}
		\delta_{\gamma_t(z)}(d\xi)\,\mu(dz)
\end{align*}
is a predictable random measure.
Consequently, by the same arguments leading to \eqref{eq:SDE reset}, the jump measure of $\widetilde X$ is $\widetilde\cJ=\cJ+\cJ^R$, and its predictable compensator is
\begin{align*}
	\widetilde\nu(dt,d\xi) =\bigl(K_t(d\xi)+K^R_t(d\xi)\bigr)\,dt.
\end{align*}
Moreover, since the reset contribution is purely discontinuous, we get
\begin{align*}
	\widetilde C_{t}=C_{t} =\int_0^t c_{s}\,ds,
\end{align*}
and \eqref{eq:second_mom_res} ensures that the reset contribution in \eqref{eq:C_res} is well defined. Indeed, since
\begin{align*}
	|\gamma^i_s(z)\gamma^j_s(z)|\le\|\gamma_s(z)\|^2,
\end{align*}
for all $i,j\in\{1,\dots,d\}$, we get, analogously to \eqref{eq:second_mom_res},
\begin{align*}
	\int_0^T h(V_{s-}) \int_{\cZ} |\gamma^i_s(z) \gamma^j_s(z)| \,\mu(dz)\,ds <\infty
\end{align*}
locally. Therefore \eqref{eq:C_res} follows from \eqref{eq:C_eta} by adding the reset term on both sides.
\end{proof}

\begin{cor}[Poissonian resetting]
Let $\widetilde X$ be as in \Cref{prop:repesent_reset}(ii) with $f(t)=\lambda e^{-\lambda t}\mathds{1}_{\{t\geq0\}}$ for $\lambda>0$ and write $r(x,z)=z_0+z\cdot x$. Set
\begin{align*}
	\gamma(x,z)\coloneqq r(x,z)-x=z_0+(z-I)x.
\end{align*}
Then the characteristics from \Cref{prop:repesent_reset}(ii) can be written as
\begin{align}
	\label{eq:B_poisson_reset}
	&\widetilde B_{t,i} =\int_0^t \widetilde{\mathfrak b}_i(\widetilde X_s)\,ds,\\
	\label{eq:C_poisson_reset}
	&\widetilde C_{t,ij} +\int_0^t\int_{\bR^d}\xi_i\xi_j\,\widetilde\nu(ds,d\xi) =\int_0^t \widetilde{\mathfrak a}_{ij}(\widetilde X_s)\,ds,
\end{align}
where $\widetilde C=C$, with $C$ as in \eqref{eq:C_eta} and
\begin{align*}
	\widetilde{\mathfrak b}_i(x) &\coloneqq \mathfrak b_i(x)+\lambda\int_{\cZ}\gamma_i(x,z)\,\mu(dz),\\
	\widetilde{\mathfrak a}_{ij}(x) &\coloneqq \mathfrak a_{ij}(x) +\lambda\int_{\cZ}\gamma_i(x,z)\gamma_j(x,z)\,\mu(dz),
\end{align*}
where $\mathfrak a$ and $\mathfrak b$ are as given in \Cref{prop:repesent_reset}(ii). Moreover,
\begin{align}
	\label{eq:nu_poisson_reset}
		\widetilde\nu(dt,d\xi)=\widetilde K_t(d\xi)\,dt,\quad
		\widetilde K_t(d\xi)=K_t(d\xi)
		+\lambda\int_{\cZ}\mathds 1_{\{\gamma(\widetilde X_{t^-},z)\neq0\}}
		\delta_{\gamma(\widetilde X_{t^-},z)}(d\xi)\,\mu(dz).
\end{align}
In particular, $\widetilde{\mathfrak b}_i\in\cP_1$ and $\widetilde{\mathfrak a}_{ij}\in\cP_2$ for all $i,j\in\{1,\dots,d\}$ and
%If $\mu$ has finite absolute moments up to order $m\ge3$, then
for every multi-index $\mathbf k\in\bN_0^d$ with $3\le|\mathbf k| $,
\begin{align}
		% \label{eq:K_poisson_reset_moments}
	\int_{\bR^d}\xi^{\mathbf k}\,\widetilde K_t(d\xi) =\sum_{|\mathbf l|=0}^{|\mathbf k|} \widetilde\alpha_{\mathbf l}\,\widetilde X_{t-}^{\mathbf l},
\end{align}
for almost every $t\ge0$, with some coefficients $\widetilde\alpha_{\mathbf l}$. Thus the Poissonian affine resetting preserves the polynomial form of the semimartingale characteristics.
\end{cor}
\begin{proof}
For Poissonian resetting we have $h\equiv\lambda$. Hence \eqref{eq:B_reset}, \eqref{eq:C_res}, and \eqref{eq:nu_reset} yield \eqref{eq:B_poisson_reset}, \eqref{eq:C_poisson_reset}, and \eqref{eq:nu_poisson_reset}. Since $\gamma_i(x,z)=z_{0,i}+\sum_{\ell=1}^d(z_{i\ell}-\delta_{i\ell})x_\ell$,
%\begin{align*}
%\gamma_i(x,z)=x_{0,i}+\sum_{\ell=1}^d(z_{i\ell}-\delta_{i\ell})x_\ell,
%\end{align*}
the map
\begin{align*}
	x\mapsto\int_{\cZ}\gamma_i(x,z)\,\mu(dz)
\end{align*}
is affine. Thus $\widetilde{\mathfrak b}_i\in\cP_1$ for $i\in\{1,\dots,d\}$. Similarly,
\begin{align*}
	x\mapsto\int_{\cZ}\gamma_i(x,z)\gamma_j(x,z)\,\mu(dz)
\end{align*}
is at most a quadratic polynomial, so $\widetilde{\mathfrak a}_{ij}\in\cP_2$. Finally, let $3\le|\mathbf k|\le m$. By \eqref{eq:nu_poisson_reset},
\begin{align*}
	\int_{\bR^d}\xi^{\mathbf k}\,\widetilde K_t(d\xi) =\int_{\bR^d}\xi^{\mathbf k}\,K_t(d\xi) +\lambda\int_{\cZ}\gamma(\widetilde X_{t^-},z)^{\mathbf k}\,\mu(dz).
\end{align*}
The first term is a polynomial in $\widetilde X_t$ of degree at most $|\mathbf k|$ by \eqref{eq:K}. Since $\gamma(\cdot,z)$ is affine, the second term is a polynomial in $\widetilde X_{t^-}$ of degree at most $|\mathbf k|$, with finite coefficients since $\mu$ has finite moments. %As $\widetilde X_t=\widetilde X_{t^-}$ for Lebesgue-a.e.\ $t$, this proves \eqref{eq:K_poisson_reset_moments}.
\end{proof}
\begin{rem}[Generators]% \label{rem:reset_generators}
Let $X$ be a polynomial diffusion process as in \Cref{prop:stochastic}(i) with extended generator $\cA$.
%and suppose that a reset mechanism is given as in \Cref{prop:repesent_reset}.
	\begin{enumerate}[(i)]
		\item For Poissonian resetting, i.e. if $f(t)=\lambda e^{-\lambda t}\mathds 1_{\{t\ge 0\}}$ for some $\lambda>0$, we have $h\equiv\lambda$. Hence the predictable compensator of $\mathsf{M}(dt,dz)$ in \eqref{eq:dyn reset X} is
\begin{align*}
	\nu^M(dt,dz)=\lambda\,\mu(dz)\,dt .
\end{align*}
In particular, along the lines of \cite{Avrachenkov12} is can be shown that $\widetilde X$ as defined in \eqref{eq:dyn reset X} is Markovian and applying It\^o's formula to $g(\widetilde X_t)$ verifies that its extended generator $\widetilde\cA$ for $g\in C^2(S)$ is
\begin{align}
	\label{eq:pois_generator}
	\widetilde{\cA} g(x)
	%=\mathfrak b(x)^\top \nabla g(x) +\frac12 \operatorname{Tr}\!\left(\mathfrak a(x)\nabla^2 g(x)\right)
	&=\cA g(x)+\lambda\int_{\cZ} \bigl(g(r(x,u))-g(x)\bigr)\,\mu(du)\\\nonumber
	&=\cA g(x)+\lambda(\cR g(x)-g(x)).
	%\nonumber &\hspace{5mm}+\lambda\int_{\cZ} \bigl(g(r(x,u))-g(x)\bigr)\,\mu(du).
\end{align}

\item The Markovianity for a general renewal reset process $\widetilde X$ can be recovered by considering the coupled process $(\widetilde X_t,V_t)_{t\geq0}$, where $V_t$ is the elapsed time since the last reset defined by \eqref{eq:time since reset}. This is a standard semi-Markov construction; see \cite{CocozzaThivent21,Cox55,Straka14}.
Indeed, let $R_t\coloneqq \tau_{N_t+1}-t$ denote the residual waiting time until the next reset.
%Conditionally on $\cG_t$, the current inter-arrival time is known
%to have survived up to age $V_t$.
Then, by the renewal property,
\begin{align*}
	\bP(R_t>s\mid\cG_t) &=\bP(\tau_{N_t+1}-\tau_{N_t}>V_t+s\mid\tau_{N_t+1}-\tau_{N_t}>V_t)\\
	&=\frac{\overline F(V_t+s)}{\overline F(V_t)} = \exp\left(-\int_{V_t}^{V_t+s}h(u)\,du\right), \quad s\ge0.
\end{align*}
Thus the conditional law of the next reset time depends on the past only through $V_t$, and between reset times, $\widetilde X$ evolves according to the Markov dynamics of $X$.
%, and at a reset time an independent mark with law $\mu$ is
%drawn and $(x,a)$ is mapped to $(r(x,z),0)$.
Consequently, the conditional future evolution given $\cG_t$ depends on the past only through $(\widetilde X_t,V_t)$, and hence $(\widetilde X_t,V_t)_{t\ge0}$ is Markovian. Applying It\^o's formula to $g(\widetilde X_t,V_t)$, using \eqref{eq:time since reset}
%$dV_t=dt-V_{t-}\int_{\cZ}\mathsf{M}(dt,du)$
and the compensator $\nu^M(dt,du)=h(V_{t-})\mu(du)dt$, gives, for $g\in C^{2,1}(S\times\mathbb R_+)$,
\begin{align*}
	\widetilde\cA g(x,a) &= \mathfrak b(x)^\top \nabla_x g(x,a) +\frac12 \operatorname{Tr}\!\left(\mathfrak a(x)\nabla_x^2 g(x,a)\right) +\partial_a g(x,a)\\
	&\hspace{5mm} +h(a)\int_{\cZ} \left( g(r(x,u),0)-g(x,a) \right)\,\mu(du).
\end{align*}
	%\item \todo{characteristics in \Cref{prop:repesent_reset} (ii) are non-polynomial}
	\end{enumerate}
\end{rem}

\section*{Declaration of generative AI and AI-assisted technologies in the manuscript preparation process}
During the preparation of this work, the authors used ChatGPT-5.6 in order to proofread the manuscript and to suggest improvements to the authors' own proofs of Lemma~2.4, Theorem 3.11, and Proposition 4.1. After using this tool, the authors reviewed and edited the content as needed and take full responsibility for the content of the published article.

\printbibliography

\end{document}